\documentclass[a4paper,11pt]{article}
\usepackage[utf8]{inputenc}   
\usepackage[T1]{fontenc}
\usepackage{lmodern}
\usepackage[english]{babel}
\usepackage{csquotes}
\usepackage{graphicx}
\usepackage{geometry}
\usepackage{comment}

\usepackage{mathtools,cases}
\usepackage{authblk}
\usepackage{bbm}
\usepackage[style=ieee, sorting=nty, backend=biber]{biblatex}
\usepackage{amsmath,amsfonts,amssymb, amsthm}
\usepackage{subfigure}
\usepackage{hyperref} 
\usepackage{bm}

\newtheorem{theorem}{Theorem}[section]
\newtheorem{corollary}{Corollary}[section]
\newtheorem{proposition}{Proposition}[section]
\newtheorem{lemma}[theorem]{Lemma}
\theoremstyle{remark}
\newtheorem{remark}{Remark}[section]

\hypersetup{colorlinks = true, allcolors=black}

\graphicspath{{./figures/}}

\newcommand{\ct}{c_{\parallel}}
\newcommand{\cn}{c_{\perp}}
\newcommand{\eit}{e_{i, \parallel}}
\newcommand{\ein}{e_{i, \perp}}
\newcommand{\et}{e_{\parallel}}
\newcommand{\en}{e_{\perp}}
\newcommand{\Xt}{\mathbf{X}}

\newcommand{\C}{\mathbf{C}}
\renewcommand*{\d}{\mathop{}\!\mathrm{d}}

\begin{document}

\title{The $N$-link model for slender rods in a viscous fluid: well-posedness and convergence to classical elastohydrodynamics equations}
\author[1,2]{François Alouges}
\author[1]{Aline Lefebvre-Lepot}
\author[3,4,5]{Jessie Levillain}
\author[6]{Clément Moreau}
\date{}

\affil[1]{Université Paris Saclay, Université Paris Cité, ENS Paris Saclay, CNRS, SSA, INSERM, Centre Borelli, F-91190, Gif-sur-Yvette, France}
\affil[2]{Institut Universitaire de France}
\affil[3]{CMAP, CNRS, École polytechnique, Institut Polytechnique de Paris, route de Saclay, 91120 Palaiseau, France}
\affil[4]{Centre National d'Etudes Spatiales, 31401 Toulouse, France}
\affil[5]{INSA Toulouse, 31400 Toulouse, France}
\affil[6]{Nantes Université, École Centrale Nantes, CNRS, LS2N, UMR 6004, F-44000 Nantes, France}

\maketitle

\begin{abstract}
Flexible fibers at the microscopic scale, such as flagella and cilia, play essential roles in biological and synthetic systems. The dynamics of these slender filaments in viscous flows involve intricate interactions between their mechanical properties and hydrodynamic drag. In this paper, considering a 1D, planar, inextensible Euler-Bernoulli rod in a viscous fluid modeled by Resistive Force Theory, we establish the existence and uniqueness of solutions for the $N$-link model, a mechanical model, designed to approximate the continuous filament with rigid segments. Then, we prove the convergence of the $N$-link model's solutions towards the solutions to classical elastohydrodynamics equations of a flexible slender rod. This provides an existence result for the limit model, comparable to those by Mori and Ohm [Nonlinearity, 2023], in a different functional context and with different methods. Due to its mechanical foundation, the discrete system satisfies an energy dissipation law, which serves as one of the main ingredients in our proofs. Our results provide mathematical validation for the discretization strategy that consists in approximating a continuous filament by the mechanical $N$-link model, which does not correspond to a classical approximation of the underlying PDE.
\end{abstract}

\textbf{Keywords:} swimming at low Reynolds number, inextensibility, filament elastohydrodynamics, $N$-link model, well-posedness, convergence.

\section{Introduction}\label{sec: intro}

Flexible fibers are ubiquitous in nature, particularly at microscopic scale, playing key roles as flagella and cilia for microbiological locomotion \cite{lauga_fluid_2020} and structural components of cell membranes, polymer chains \cite{broedersz2014modeling}, and micro-robotics \cite{bente2018biohybrid}. 

The dynamics of a slender filament in a fluid is governed by the coupling between mechanical properties of the deformable filament and hydrodynamic interactions between the filament and the fluid. In addition, internal or external effects such as gravity, magnetic field, or internal activity can produce a broad range of behaviors like undulating, twisting, or knotting. 

Casting this fluid-structure interaction problem as a set of equations requires various modeling assumptions \cite{du2019dynamics,lindner2015elastic}, going from a full 3D description of both filament and fluid to simplifications pertaining to the filament slenderness and the specificities of low-Reynolds number hydrodynamics. In the slender filament limit, for planar motion, and neglecting the effects of internal shear, elastic restoring torque is linearly related to local curvature and bending stiffness, according to Euler-Bernoulli beam theory \cite{antman_nonlinear_2005, tornberg2004simulating}.

For the treatment of hydrodynamics, Resistive Force Theory (RFT) \cite{gray1955propulsion}, developed in the 1950s and which approximates local drag as a linear anisotropic operator related to local velocity, remains to this day a prominently popular choice for micro-filament modelling and simulation, offering a simple approximation with satisfying accuracy \cite{yu2006experimental}. More complex and non-local models, often termed as slender body theory (SBT), and regularized Stokeslet methods typically provide higher accuracy at the price of more involved computations \cite{tornberg2004simulating,nazockdast2017fast}.

In this article, we focus on a 1D, planar, inextensible Euler-Bernoulli deformable rod in a viscous fluid modeled through RFT. 
It leads to a fourth-order nonlinear partial differential equation (PDE) system that is considered standard in the literature on filament elastohydrodynamics \cite{hines1978bend,antman_nonlinear_2005,lauga_fluid_2020}.

The nonlinear terms arising from the inextensibility constraint make the filament elastohydrodynamics equations notoriously tricky to solve numerically with reasonable levels of accuracy and computational efficiency \cite{gadelha2010nonlinear}. In parallel of classical PDE discretization methods, mechanical discrete models have been proposed, based on replacing the continuous elastic body with a collection of rigid parts connected by elastic junctions. Common examples include the $N$-beads formulation \cite{maxian2022hydrodynamics}, for which the filament is seen as a chain of spheres, and the $N$-link, for which it is seen as a chain of slender straight rods \cite{alouges_self-propulsion_2013}. It is worth noting that these models not only constitute approximations of a flexible fiber, but also faithfully describe the structure of a certain type of flexible micro-robots built as an assembly of magnetic parts \cite{alouges2015can,moreau2018asymptotic,pauer2021programmable}. 

The $N$-link model, considered in this article, relies on analytic integration of the hydrodynamic force density given by RFT, carried out on individual segments; in turn, the dynamics is reduced to a first-order differential-algebraic equation system. This powerful approach allows further modelling refinements on dealing with hydrodynamics \cite{walker2020efficient} and obstacles \cite{elgeti2016microswimmers,spagnolie2023swimming}. 

Whether it is on the PDE formulation or its discrete $N$-link approximation, the mathematical analysis of elastohydrodynamics is relatively scarce, with the first well-posedness result in the continuous case (existence and uniqueness of the solution for a given initial data) having been stated only recently by Mori \& Ohm \cite{mori_well-posedness_2023}. Relying on the Banach fixed-point theorem, the authors establish global existence of solutions for small initial data and local existence of solutions for arbitrary initial data. On the other hand, justification of existence and uniqueness of the solutions of the $N$-link system is currently lacking. 

Furthermore, the approximation of a flexible fiber by a collection of small rigid segments has reasonable physical grounds and we observe numerical convergence towards the continuous model \cite{moreau2018asymptotic}. However, no formal proof of convergence as the number $N$ of links tends to infinity is available to the best of our knowledge. 

The objective of the present paper is to address both of these questions. Hence, we establish the well-posedness of the $N$-link equations (Theorem \ref{thm: discrete well-posedness}), and the convergence (up to extraction of a subsequence) of this solution in suitable functional spaces towards the solution of the continuous elastohydrodynamics equations (Theorem \ref{thm: convergence}). 
Considered together, Theorems \ref{thm: discrete well-posedness} and \ref{thm: convergence} imply the global existence of solutions for the elastohydrodynamic PDE (Corollary \ref{corollary}).

Both proofs rely on classical arguments. Theorem \ref{thm: discrete well-posedness} for free boundary conditions is an application of the Cauchy-Lipschitz theorem, which requires verifying that the hydrodynamic resistance matrix is invertible, and the use of an energy dissipation estimate (Proposition \ref{thm: energy}) to ensure boundedness of the solutions. Extension to other standard boundary conditions is also discussed. For Theorem \ref{thm: convergence}, in order to connect the finite- and infinite-dimensional variables of both systems, we define piecewise-constant, and continuous-piecewise-affine interpolates of the $N$-link system variables (position, orientation, forces and moments). We derive uniform bounds in $N$ for each of them, again mostly relying on energy dissipation, and prove convergence using compactness methods. 

The paper is structured as follows. In Section \ref{sec: formulation}, we describe the continuous model, the coarse-grained $N$-link model, and state the two main results. Section \ref{sec: energy} is dedicated to establishing the energy dissipation formula. The proofs of Theorems \ref{thm: discrete well-posedness}, \ref{thm: convergence} and Corollary \ref{corollary} are presented Section \ref{sec: proofs}. Finally, we discuss a list of possible model extensions and a few open problems in Section \ref{sec: discussion}.

\begin{figure}[ht!]
    \centering
    \includegraphics[width=\textwidth]{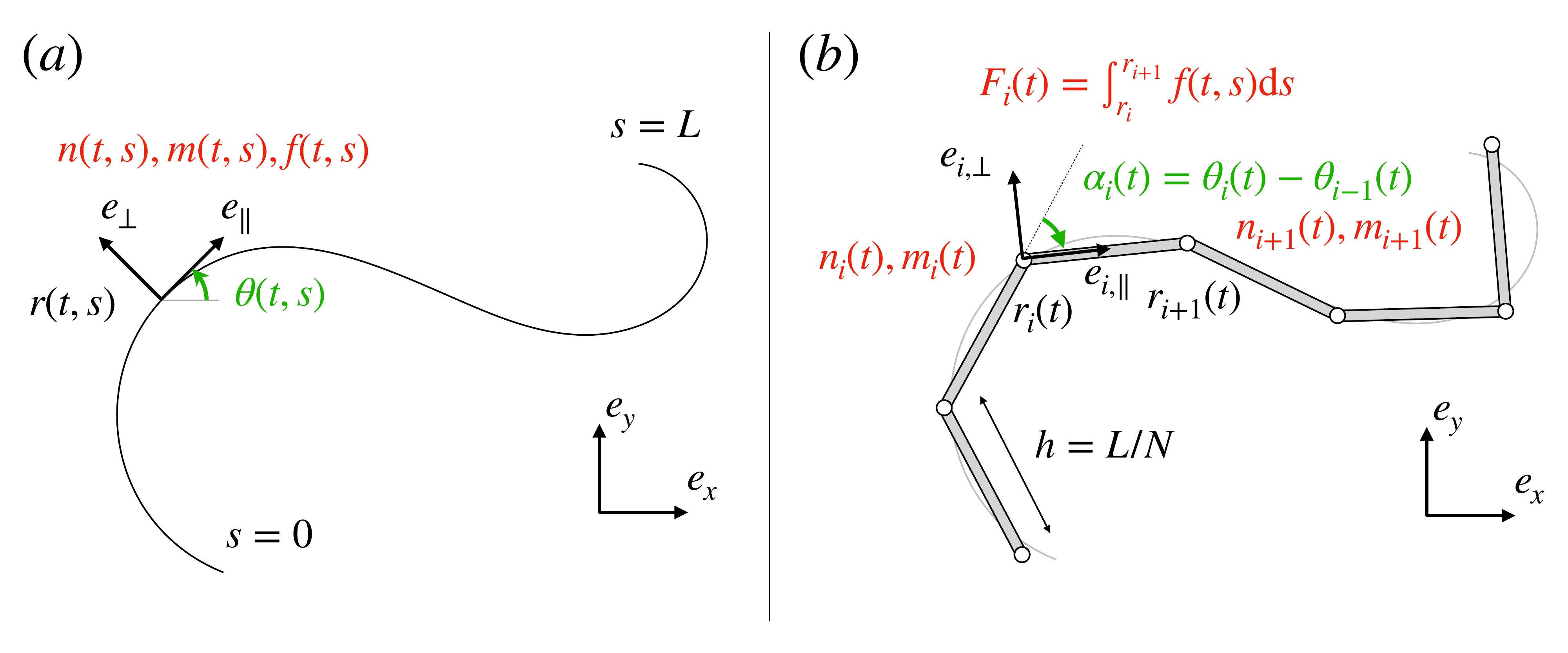}
    \caption{Diagram and notations for (a) the continuous elastohydrodynamic flagellum model and (b) the $N$-link model.}
    \label{fig:Nlink}
\end{figure}

\section{Problem formulation and main results}\label{sec: formulation}

\subsection{Continuous model} \label{sec: continuous model}

We consider a filament of length $L>0$ undergoing planar deformations in a fluid at a low Reynolds number, parametrized by 
\[
r:\left\{\begin{array}{ll}\mathbb{R}_+\times [0,L] \to \mathbb{R}^2 \subset \mathbb{R}^3\\ (t,s) \mapsto r(t,s)\end{array}
    \right.
\]
where $s$ is the filament arclength and $t$ the time, as shown in Figure \ref{fig:Nlink}(a). The filament is assumed to be inextensible and unshearable. We call  $n(t,s)$ and $m(t,s)$ the contact forces and moments inside the filament,  and $f(t,s)$ the external force density due to the fluid. The filament dynamics is then classically governed by the following system of equations \cite{antman_nonlinear_2005}
\begin{align}
    \left\{\begin{array}{ll}
 & n_s +f = 0,\\
 & m_s + r_s \times n = 0, \\
 & m = E \theta_s e_z =  E r_{ss} \times r_s,\\ 
 & \vert r_s\vert ^2 = 1,
    \end{array}
    \right.
    \label{eq: continuous antman}
\end{align}
where $\theta(t,s)$ is defined as the angle between the $x$ axis and the tangent vector to the filament $\et$, with $\et = r_s(t,s) := \displaystyle \frac{\partial r}{\partial s}(t,s)$. Assumed to be differentiable, $\theta$ is uniquely determined on $[0,L]$ once a representative has been chosen at $s=0$. 
The first two equations of \eqref{eq: continuous antman} reflect the force and torque balance on the filament. 
The third equation is the constitutive materical equation linking contact moments $m$ to the angle $\theta$. Here, we use a neo-Hookean description and the bending moment is therefore linearly related to the local curvature with bending stiffness $E>0$. Finally, the last equation encapsulates the inextensibility constraint.

\begin{remark}
\label{remark: external torque}
    In Antman \cite{antman_nonlinear_2005}, an additional term $\ell$ appears in the second equation of System \eqref{eq: continuous antman}, accounting for external torque density, with a typical example of such effects being the torque induced by a magnetic field on a magnetized rod \cite{alouges2015can}. Here, no such effects are considered, and there is no torque due to hydrodynamic drag, because the motion is planar \cite{chwang1974hydromechanics,walker2020efficient}.   

\end{remark}

Using Resistive Force Theory \cite{gray1955propulsion}, the density of external fluid forces can be modeled as:
\begin{equation}
    f(t, s) = -\cn(\en \cdot \dot r)\en- \ct (\et \cdot \dot r)\et = \C (\theta(t,s)) \dot r,
    \label{eq: RFT}
\end{equation}
 where $\dot r (t,s)$ is defined as $\frac{\partial r}{\partial t}(t,s)$, $\et=(\cos\theta,\sin\theta)^\intercal$, $\en=(-\sin\theta,\cos\theta)^\intercal$. The parallel and perpendicular drag coefficients, respectively $\ct$ and $\cn$, are non-negative with $\ct \neq \cn$.  Finally,
\begin{equation*}
\C(\theta) = - \left[ \cn \begin{pmatrix}
    -\sin \theta \\ \cos \theta
\end{pmatrix} \begin{pmatrix}
    -\sin \theta \,,\, \cos \theta
\end{pmatrix}  + \ct \begin{pmatrix}
    \cos \theta \\ \sin \theta
\end{pmatrix} \begin{pmatrix}
    \cos \theta \,,\, \sin \theta
\end{pmatrix}\right]
   \label{eq: C}
\end{equation*}
 is a negative definite $2\times 2$ matrix, depending regularly on the variable $\theta$. In a slight abuse of notation, it will sometimes be convenient to rewrite $\C(\theta)$ as a function of $u=(\cos(\theta),\sin(\theta))^\intercal$:
\begin{equation*}
    \mathbf{C}(u) = -\left(\cn I + (\ct-\cn)u u^\intercal\right).
    \label{eq: Cu}
\end{equation*}

To complete the description of the model, it remains to set boundary conditions at $s=0$ and $s=L$ for any time $t>0$. At $s=L$, we assume a free end, leading to Dirichlet boundary conditions for $n$ and $m$: 
\begin{equation}
    n(t,L)=0, \quad m(t,L)=0.
    \label{eq: BC distal}
\end{equation}
At $s=0$, we investigate three possible cases:
\begin{itemize}
    \item \textit{free} end, identical to the other end:  
    \begin{equation}
        n(t, 0)=0, \quad m(t, 0)=0,
        \label{eq: BC proximal free}
    \end{equation}
    \item \textit{pinned} end, for which the position of the filament at $s=0$ is fixed in space, but its orientation $r_s(t,0)$ is free to change:
    \begin{equation}
        \dot r(t,0) = 0, \quad  m(t,0)=0,
        \label{eq:BC proximal pinned}
    \end{equation}
    \item \textit{clamped} end, for which the extremity has a fixed position and orientation:
    \begin{equation} 
        \dot r(t,0) = 0, \quad \dot  \theta(t, 0) = 0.
        \label{eq: BC proximal clamped}
        \end{equation}
\end{itemize}
We finally present the corresponding full elastohydrodynamics system of equations in the case of free boundary conditions at both ends as
\begin{align}
    \left\{\begin{array}{ll}
 & n_s + \C(r_s) \dot r = 0,\\
 & m_s + r_s \times n = 0, \\
 & n(t,0)=n(t,L)=0, \\
 & m(t,0)=m(t,L)=0, \\
 & m = E r_{ss} \times r_s = E \theta_s e_z ,\\ 
 & \vert r_s\vert ^2 = 1.
    \end{array}
    \right.
    \label{eq: continuous}
\end{align}
Extension to pinned and clamped boundary conditions can be obtained simply by changing the third and fourth lines and straightforwardly adapting the proofs. Notice that this system has to be supplied with an initial condition $r^0$ for $r$ at $t=0$.

Existence and uniqueness of solutions to System \eqref{eq: continuous} has recently been established by Mori and Ohm \cite[Theorem 1.1]{mori_well-posedness_2023}, although in a different formulation, based on the curvature variable $\kappa = r_{ss}$, and involving different functional spaces than the ones we use in this paper. We provide a detailed discussion on this matter at the end of this section.

\subsection{{$N$-link model}}\label{sec: defNlink}

We now introduce the discrete filament model, which can be interpreted as a coarse-grained version of a continuous flexible filament. This model is often called the $N$-link system \cite{moreau2018asymptotic, alouges_self-propulsion_2013} in the context of microscopic locomotion.   The $N$-link filament, represented in Figure \ref{fig:Nlink}(b), is composed of $N$ rigid segments (or ``links'') of size $h =  L/N$ connected by torsional springs and surrounded by a fluid at zero Reynolds number.  The extremities of the $i$-th segment, for $1 \leq i \leq N$, are at positions $r_i(t)$ and $r_{i+1}(t)$, with $(r_i(t))_{1\leq i\leq N+1}$ given in $\mathbb{R}^2\subset\mathbb{R}^3$. From this, the position of the point of arclength $s$ on the discrete filament can be expressed as
\begin{equation}
r^h(t,s) = r_i(t) + \Big(s-(i-1)h\Big)\left(\frac{r_{i+1}(t) -r_i(t)}{h}\right),
\label{eq: interp_rh}
\end{equation}
for $s\in L_i=[(i-1)h, ih]$ and $1\leq i\leq N$, which is the linear interpolate between $r_i(t)$ and $r_{i+1}(t)$ on the $i$-th segment.

Due to the inextensibility condition we have $|r^h_s| = 1$ which translates into
\begin{equation}
\displaystyle \frac{\vert r_{i+1}(t) - r_i(t) \vert}{h} = 1,\quad 1 \leq i \leq N\,.
\label{eq: inext}
\end{equation}
We can then define the unit vector parallel to the $i$-th segment as 
\begin{equation}
\eit(t) = \displaystyle \frac{r_{i+1}(t)-r_i(t)}{h} = \begin{pmatrix}
    \cos \theta_i(t) \\ \sin \theta_i(t)
\end{pmatrix}, 1\leq i \leq N
\label{eq:eit}
\end{equation}
where we call, for each link, $\theta_i$ an angle between $e_x$ and $\eit$, defined modulo $2 \pi$.
The choice of the representative $\theta_i$ is immaterial for this geometric description but matters later when introducing the elastic torques. In fact, once a choice $(\theta_i^0)_{1\leq i\leq N}$ of $(\theta_i)_{1\leq i\leq N}$ is made at $t=0$, its representative at all times will be uniquely determined by dynamics. This is a fundamental difference between the discrete and continuous systems that is further discussed in Remark~\ref{todo}.

For $2\leq i \leq N$,  we denote by $n_i(t)$,  $m_i(t)$ the contact force and moment exerted by $L_i$ on $L_{i-1}$. The forces and moments at both extremities $n_1,\, m_1,\, n_{N+1},\, m_{N+1}$ are given by the boundary conditions, which mirror those used for the continuous filament (Eqs. \eqref{eq: BC distal}-\eqref{eq: BC proximal clamped}): the extremity of the $N$-th segment is left free, which translates to $m_{N+1}(t)=0$ and $n_{N+1}(t)=0$, while for the first segment, we consider the following three cases:

\begin{itemize}
    \item \textit{free} end:  
    \begin{equation*}
        n_1(t)=0, \quad m_1(t)=0;
        \label{eq: BC proximal free disc}
    \end{equation*}
    \item \textit{pinned} end:
    \begin{equation*}
        \dot r_1(t) = 0, \quad m_1(t)=0,
        \label{eq:BC proximal pinned disc}
    \end{equation*}
    leaving $n_1(t)$ unknown;
    \item \textit{clamped} end:
    \begin{equation*} 
        \dot r_1(t) = 0, \quad \dot \theta_1(t)=0,
        \label{eq: BC proximal clamped disc}
        \end{equation*}
        leaving both $n_1(t)$ and $m_1(t)$ unknown.
\end{itemize}

Then, denoting by $f_{\mathrm{ext},i}(t)$ the drag force exerted by the fluid on the $i$-th segment and using Equation \eqref{eq: RFT}, one has

\begin{eqnarray}
   f_{\mathrm{ext},i}(t) &=& \displaystyle\int_{L_i} \C(\theta_i(t))\,\, \dot r^h(t,s)\,\,\d s\nonumber\\
   &=& h\,\, \C(\theta_i(t))\,\,\dot r_{i+\frac12}(t)\,,
    \nonumber
\end{eqnarray}
where $ r_{i+1/2} = \displaystyle \frac{r_i + r_{i+1}}{2}$. 
The force balance on the $i$-th segment is therefore given by 
\begin{equation}
  h\, \C(\theta_i(t))\,\dot r_{i+\frac12}(t) + n_{i+1}(t) - n_i(t) = 0, \quad 1\leq i \leq N.
  \label{eq: Force balance}
\end{equation}

Similarly, the drag torque from the fluid, with respect to the origin of the reference frame, can be computed as (omitting the dependence in time for brevity)
\begin{equation}
    \begin{array}{lll}
        m_{\mathrm{ext}, i} & \displaystyle =  \int_{L_i} r^h \times \C(\theta_i)\,\dot r^h\,\d s\\
        &\displaystyle = \int_{L_i} (r^h- r_{i+1/2}) \times \C(\theta_i)\,(\dot r^h - \dot r_{i+1/2})\,\d s + h \,r_{i+1/2}\times \C(\theta_i)\, \dot r_{i+1/2}\\
        & \displaystyle = \int_{L_i} (r_{i+1}-r_i) \times \C(\theta_i)\,(\dot r_{i+1} - \dot r_{i})\left(\frac{1}{2}+\frac{s-(i-1)h}{h}\right)\,\d s + h\, r_{i+1/2}\times \C(\theta_i)\, \dot r_{i+1/2}\\
    & \displaystyle = - \frac{h^3}{12}c_{\perp} \dot \theta_i e_z + h \,r_{i+1/2}\times \C(\theta_i)\, \dot r_{i+1/2},
    \end{array}
    \nonumber
\end{equation}
and the torque balance equation on the $i-$th segment 
then reads
\[
    - \frac{h^3}{12}c_{\perp} \dot \theta_i e_z + h\, r_{i+1/2}\times \C(\theta_i)\, \dot r_{i+1/2} + m_{i+1}-m_i + r_{i+
        1}\times n_{i+1}-r_i\times n_i = 0\,. 
\]
Using the force balance equation \eqref{eq: Force balance}, together with the definition of $\eit$, this simplifies to

\begin{equation}
    m_{i+1}-m_i + h\,\eit\times \frac{n_{i+1}+n_i}{2} - \frac{h^3}{12}c_{\perp}\dot \theta_i e_z = 0,\quad 1\leq i \leq N\,.
    \label{eq: Moment balance}
\end{equation}

Finally, we model the relation between the angles and the elastic torque as
\begin{equation}
   m_{i}(t) = \gamma (\theta_{i}(t) - \theta_{i-1}(t))e_z ,\quad 2 \leq i \leq N\,,
    \label{eq: m}
\end{equation}
where the stiffness of the elastic junctions is defined as $\gamma = E/h$ in order to be consistent with the previous continuous model.

\begin{remark}\label{todo}

Here, there is a subtle but important difference between the continuous description and the $N$-link model in the way $\theta$ is deduced from $r$ at $t=0$.
Indeed, in the continuous case, the choice of $r(0,\cdot)$ naturally defines $\theta_s(0,\cdot)$, which in turn determines a unique differentiable $\theta(0,\cdot)$, provided that one representative $\theta(0,0)$ has been chosen. On the other hand, in the $N$-link model, the local representative of $(\theta_i)_i$ is \textit{not} uniquely determined from the knowledge of $(r_i)_{i}$ and $\theta_1$ at $t=0$. Therefore, a global choice of representatives at $t=0$ for each of the $(\theta_i)_i$ is compulsory. In particular, this choice of local representative matters in computing the local moment from \eqref{eq: m}. 

\end{remark}

Gathering Eqs. \eqref{eq: inext}, \eqref{eq: Force balance}, \eqref{eq: Moment balance} and \eqref{eq: m} with free boundary conditions, we finally obtain:
\begin{equation}
    \left\{\begin{array}{llll}
   n_{i+1} - n_i + h \,\C(\theta_i)\, \dot r_{i+1/2}= 0, & 1 \leq i \leq N,\\
    \displaystyle m_{i+1}-m_i + h \,\eit \times \frac{n_{i+1}+n_i}{2} - \frac{h^3}{12}\cn \dot \theta_ie_z = 0, & 1 \leq i \leq N,\\
     n_1 = n_{N+1}=0,\\ 
    m_1=m_{N+1}=0,\\
    m_{i} = \displaystyle\frac{E}{h} (\theta_{i} - \theta_{i-1})e_z, & 2 \leq i \leq N,\\
    \displaystyle \frac{\vert r_{i+1}-r_i\vert}{h}=1, & 1 \leq i \leq N,
    \end{array}
    \right.
    \label{eq: Nlink_bis_energ}
\end{equation}
where the variables $(r_i)_{1\leq i\leq N+1}$ and $(\theta_i)_{1\leq i\leq N}$ are coupled by~\eqref{eq:eit}.  Recall also that the system can be adapted for pinned or clamped boundary conditions, by replacing the third and fourth lines with the corresponding boundary conditions on the first segment.

As already mentioned in Remark \ref{todo}, here one needs to give both an initial condition for $(r_i)_{1\leq i\leq N+1}$ and $(\theta_i)_{1\leq i\leq N}$ still coupled by~\eqref{eq:eit}, while for the continuous system, the initial condition on $\theta$ could be deduced from the one for $r$. 

In the following, we establish that System \eqref{eq: Nlink_bis_energ} has a unique solution for any given initial condition. However, this solution does not prevent self-intersection of the filament. This is also true for the continuous system \eqref{eq: continuous}. In reality, self-intersection is impossible for 2D motion, and is usually enforced in models e.g. by solving Stokes equations more accurately, adding short-range repulsion forces or Lagrange multipliers. 

Several terms in System \eqref{eq: Nlink_bis_energ} can be seen as discretized counterparts of corresponding terms in the continuous elastohydrodynamics equations \eqref{eq: continuous}. The space derivatives $n_s$, $m_s$, $r_s$ and $\theta_s$ in \eqref{eq: continuous} are immediately identified to forward finite differences in \eqref{eq: Nlink_bis_energ}. Then, the terms $\dot{r}$ and $r_s \times n$ in \eqref{eq: continuous} can be matched to the terms of half-sum type ($r_{i+1/2}$ and $(n_{i+1}+n_i)/2$) in \eqref{eq: Nlink_bis_energ}. Finally, the second equation in \eqref{eq: Nlink_bis_energ} features an additional term, $-\frac{h^3}{12}\cn \dot \theta_ie_z$, that cannot be matched to the moment balance in \eqref{eq: continuous}, because it accounts for the hydrodynamic torque on a rigid link of positive length $h$. Heuristically, one can expect that this term vanishes when $h$ goes to zero, which is in agreement with the fact that external moments are neglected in \eqref{eq: continuous} as stated in Remark \ref{remark: external torque}.

Nonetheless, the presence of this supplementary term highlights the important fact that the equations governing the dynamics of the $N$-link model cannot be obtained as any direct \textit{numerical} discretization of System \eqref{eq: continuous}. In a sense, they rather arise from a \textit{mechanical} discretization of a continuous rod. This means that the mathematical convergence of one model to another, although physically plausible and numerically verified \cite{moreau2018asymptotic}, is not obvious, and this is the purpose of the present study. 

\subsection{Matrix expression}

In order to proceed, we rewrite \eqref{eq: Nlink_bis_energ} as an differential-algebraic matrix system. Let us introduce the vectors $\Xt = \left( \theta_1, \, \dots, \, \theta_N, \, r_1^x, r_1^y \right)^\intercal$, $\mathbf{N} = (n_1^x, \, n_1^y,\, \dots, \, n_N^x, \, n_N^y)^\intercal$, and $\mathbf{M} = (m_1^z, \, \dots ,\,m_N^z)^\intercal$, where, for any $v \in \mathbb{R}^3$, we have denoted by $v^x$, $v^y$, $v^z$ its coordinates in the $e_x$, $e_y$ and $e_z$ directions. We now rewrite the system \eqref{eq: Nlink_bis_energ} using $\mathbf{X}, \mathbf{N}, \mathbf{M}$ as unknowns.

Notice that the boundary conditions $n_{N+1}$ and $m_{N+1}^z$ do not appear among the unknowns in $\mathbf{N}$ and $\mathbf{M}$ due to the free boundary conditions, while we keep $m_{1}^z$ and $n_{1}$ since we might consider different boundary conditions at the first end. The $(r_i)_{2\leq i \leq N+1}$ were also removed from the unknowns, knowing that, due to the definition \eqref{eq:eit} of $\eit$, $r_i$ can be expressed in terms of $r_1$ and $(\theta_k)_{1\leq k\leq N}$ as
\begin{equation}
r_i = r_1 + h \sum\limits_{k=1}^{i-1} \begin{pmatrix}
    \cos \theta_k \\ \sin \theta_k
\end{pmatrix},\quad 2\leq i \leq N+1.
    \label{eq: Ri}
\end{equation}

Let us now rewrite \eqref{eq: Nlink_bis_energ} in a matrix form.  First, by differentiation of \eqref{eq: Ri}, one has $\dot r_{i+1/2} = (G(\Xt)\dot \Xt)_i$, where the matrix $G(\Xt)$, of size $(2N, N+2)$, is such that for any $\mathbf{W} = (W_i)_i \in \mathbb{R}^{N+2}$ and $1\leq i \leq N $, 
\begin{equation}
(G(\Xt)\mathbf{W})_i = \begin{pmatrix}
    W_{N+1}\\ W_{N+2}
\end{pmatrix}+ h \sum\limits_{k=1}^{i-1}\begin{pmatrix} -\sin \theta_k \\ \cos  \theta_k \end{pmatrix} W_k + \frac{h}{2}\begin{pmatrix} -\sin \theta_i \\ \cos  \theta_i \end{pmatrix} W_i \in \mathbb{R}^2.
    \label{eq:G}
\end{equation}

As a consequence, we rewrite  \eqref{eq: Nlink_bis_energ} as:
\begin{equation}
\left(
\begin{array}{ccc}
A_{11} & A_{12}(\Xt) & 0_{2N,N}\\
A_{21}(\Xt) & A_{22} & A_{23}\\
0_{N,2N} & 0_{N,N+2} & I_N
\end{array}
\right)
\left(
\begin{array}{c}
\mathbf{N}\\ \dot \Xt\\ \mathbf{M}
\end{array}
\right)= 
\left(
\begin{array}{c}
F_1\\ F_2\\ F_3(\Xt)
\end{array}
\right)\,,
\label{eq: matrixAXF}
\end{equation}
where, using
\begin{equation}
\mathbb{C}(\Xt) = \text{diag}(\C(\theta_1), \, \dots, \, \C(\theta_N)), \label{eq:CC}
\end{equation}
the underlying matrices are defined by
$$A_{11} = \left(
\begin{array}{ccccc}
-I_2 & I_2 & \cdots & \cdots & 0_2\\
0_2 & -I_2 & I_2 & \cdots & 0_2\\
\vdots & \ddots  & \ddots & \ddots &\vdots\\
\vdots & \ddots  & \ddots & \ddots &\vdots\\
0_2 & \cdots  & \cdots &0_2 & -I_2
\end{array}
\right),\,\,A_{12}(\Xt) = h \,\mathbb{C}(\Xt)G(\Xt)\,,
$$
$$
A_{21}(\Xt) = \frac{h}{2}\left(
\begin{array}{ccccc}
e^\intercal_{1, \perp} & e^\intercal_{1, \perp} & 0_2 & \cdots & 0_2\\
0_2 & e^\intercal_{2, \perp} & e^\intercal_{2, \perp} & \cdots & 0_2\\
\vdots & \ddots  & \ddots & \ddots &\vdots\\
0_2 & \cdots  & \cdots &e^\intercal_{N, \perp}  & 0_2\\
I_2 & 0_2  & \cdots & \cdots & 0_2\\
\end{array}
\right),\,\,A_{22} = \displaystyle -\frac{h^3}{12}\cn \begin{pmatrix}
1 &  &  &  & \\
 & \ddots  &  &  & \\
 &  & 1 &  & \\
 &  &  & 0 & \\
 &  &  &  & 0
\end{pmatrix}\,,
$$
$$
A_{23} = \left(
\begin{array}{ccccc}
-1 & 1 & 0 & \cdots & 0\\
0 & -1 & 1 & \cdots & 0\\
\vdots & \ddots  & \ddots & \ddots &\vdots\\
0 & \cdots  & \cdots &-1  & 0\\
0 & \cdots & \cdots & \cdots & 0\\
\end{array}
\right),
$$
and
$$
F_1 = 0_{2N},\,\,F_2=0_{N+2},\,\,F_3(\Xt)=\frac{E}{h}(0,\theta_2 - \theta_1,\cdots,\theta_N-\theta_{N-1})^\intercal\,.
$$
Note that the $(\ein)_{1\leq i \leq N}$ in $A_{21}$ stem from the fact that for any $v\in \mathbb{R}^2\subset \mathbb{R}^3$ and $1 \leq i \leq N$, $\eit \times v \cdot e_z = \ein \cdot v$. The system~\eqref{eq: matrixAXF} needs to be complemented with initial conditions for $\Xt$ i.e. for $(\theta_i)_{1\leq i\leq N}$ and $r_1$.

Through the use of~\eqref{eq: Ri}, the differential-algebraic system~\eqref{eq: matrixAXF}, complemented with the conditions $n_{N+1}(t)=0$ and $m_{N+1}^z(t)=0$ for all $t\geq 0$, is evidently equivalent to the original problem~\eqref{eq: Nlink_bis_energ}.

\subsection{Main results}

The aim of this paper is to show that the solution of the discrete system \eqref{eq: matrixAXF} (or equivalently \eqref{eq: Nlink_bis_energ}) is well defined and converges to the solution of the continuous system \eqref{eq: continuous}, as $h$ goes to 0.  

For the sake of completeness, let us recall the definitions of the (classical) functional spaces which are used in this section. Let $\Omega$ be an open set of $\mathbb{R}^d$. We define 
\[
H^1(\Omega)=\left\{f\in L^2(\Omega) \, / \, \partial_i f \in L^2(\Omega) \text{ for } 1\leq i \leq d \right\},\, \quad \|f\|_{H^1(\Omega)}^2 = \|f\|_{L^2(\Omega)}^2 + \sum_{i} \|\partial_i f\|_{L^2(\Omega)}^2,
\]
\[
H^2(\Omega)=\left\{f\in H^1(\Omega)  \, / \,   \partial_{i}\partial_j f \in L^2(\Omega) \text{ for } 1\leq i, j \leq d \right\},\, \quad \|f\|_{H^2(\Omega)}^2 = \|f\|_{H^1(\Omega)}^2 + \sum_{i,j} \|\partial_i \partial_j f\|_{L^2(\Omega)}^2,
\]
where $\partial_i$ denotes the derivative with respect to the $i$-th variable in $\mathbb{R}^d$. If $T>0$ is given and $H$ is a Hilbert space of functions defined from $\Omega$ to $\mathbb{R}$, we also define the space of functions $L^2$ in time with values in $H$ :
\[
L^2(0,T;H)=\left\{ f: [0,T]\times \Omega \to \mathbb{R}\, / \, \int_0^T \| f(t,\cdot)\|_H^2 \, \d t < +\infty \right\},
\]
endowed with the norm 
\[
\| f\|_{L^2(0,T;H)} = \left( \int_0^T \| f(t,\cdot)\|_H^2 \, \d t \right)^{1/2}.
\]
Finally, if $f$ is a vector-valued function, the notation $f \in H$ (resp. $f \in L^2(0,T;H)$) means that each component of $f$ belongs to $H$ (resp. $L^2(0,T;H)$).

\begin{theorem}[Well-posedness for the $N$-link swimmer]
    \label{thm: discrete well-posedness}
    The discrete system \eqref{eq: matrixAXF}, given a set of initial conditions $r_1(0)=r_1^0 \in \mathbb{R}^{2}, \, \theta_i(0)=\theta_i^0 \in \mathbb{R}$ for $1\leq i \leq N $, admits a unique global solution $r_1 \in C^1(\mathbb{R}_+), \, (\theta_i)_{1\leq i \leq N}\in C^1(\mathbb{R}_+)^{N}$, $(m_i^z)_{1\leq i \leq N}\in C^1(\mathbb{R}_+)^{N},\, (n_i)_{1\leq i \leq N}\in C^1(\mathbb{R}_+)^{N}$.
\end{theorem}

For $i=1,\cdots,N+1$, let $\phi_i$ be continuous and piecewise linear functions, defined in $H^1(0, L) $ by $\phi_i(jh) = \delta_{i-1,j}$ for $j=0,\cdots , N$ and $\mathcal{V}_h := \text{Span}((\phi_i)_{1\leq i\leq N+1})= \{v_h \in C^0(0,L) \, | \,  v_h \text{ affine on } [(i-1)h, ih] \text{ for } 1 \leq i \leq N\}$. 
We then introduce the following linear interpolates in $\mathcal{V}_h$, for $T>0$ and
$(t,s) \in Q_T = [0,T]\times[0,L]$
\begin{align}
    r^h(t,s) &= \sum\limits_{i=1}^{N+1} r_i(t) \phi_i(s), \label{eq: interp_r_1}\\
    m^{h,z}(t,s) &= \sum\limits_{i=1}^{N+1} m_i^z(t) \phi_i(s), \\
    n^h(t,s) &= \sum\limits_{i=1}^{N+1} n_i(t) \phi_i(s).\label{eq: interp_n}
\end{align}
Note  that \eqref{eq: interp_r_1} is nothing but a rewriting of \eqref{eq: interp_rh}. We also highlight that $r^h_s(t,s) = \displaystyle \frac{r_{i+1}(t)-r_i(t)}{h}$ for $s \in [(i-1) h, i h]$, from which we deduce that $\vert r^h_s(t,s) \vert^2 = 1$ by construction. Now, let us define the piecewise constant interpolate
\begin{equation}
\bar \theta^h(t,s) = \sum\limits_{i=1}^N \theta_i (t) \mathbbm{1}_{L_i}(s). \label{eq: interp_theta}
\end{equation}
which belongs to $\mathcal{W}_h = \{w_h \in L^2(0,L) \, | \,  w_h \text{ constant on } ((i-1)h, ih) \text{ for } 1 \leq i \leq N\}$.

For technical reasons, we also introduce a piecewise linear interpolate for $\theta$ in $\mathcal{V}_h$ and a piecewise constant interpolate for $m$ in $\mathcal{W}_h$:
 \begin{align}
    \theta^h(t,s) & = \sum\limits_{i=1}^{N} \theta_i(t) \phi_{i+1}(s) + \theta_1(t) \phi_1(s), \label{eq: interpolate theta_h} \\
  \bar{m}^{h,z}(t,s) &= \sum\limits_{i=1}^N m_{i}^z (t)\mathbbm{1}_{L_i}(s).
   \label{eq: interp_r_2}
\end{align}
Eventually, we define
\begin{equation*}
   \hat r^h(t,s)= r^h(t,0)  + \int_0^s (\cos \theta^h(t,u),\, \sin \theta^h(t,u))^\intercal \d u,
\end{equation*}
in such a way that $\hat r^h_s = (\cos \theta^h,\, \sin \theta^h)^\intercal$. As $\theta^h$ is continuous and piecewise linear, this last expression is differentiable in $s$ and we have
\begin{equation*}
    \hat r^h_{ss} = \theta^h_s (-\sin \theta^h, \, \cos \theta^h)^\intercal,
    \label{eq: rssthetastheta}
\end{equation*}
which gives in particular 
\begin{equation*}
   \bar m^{h,z} =  E (\hat r^h_{ss} \times \hat r^h_s)\cdot e_z = E \theta_s^h.
\label{eq: rssrsthetas}
\end{equation*}

We can now state the second main result.

\begin{theorem}[Convergence]
    \label{thm: convergence}
Let $r^0\in C^2(0,L)$ be given together with a representative $\theta^0\in C^1(0,L)$  such that $r^0_s = (\cos \theta^0, \, \sin \theta^0)^\intercal \in C^1(0,L)$.

Let $T>0$. For all $N \in \mathbb{N}^*$, let a set of discrete initial conditions $\Xt^0=(\theta_1^0, \, \dots,\, \theta_N^0,\, r_1^{x,0},\, r_1^{y,0}) \in \mathbb{R}^{N+2}$ be given together with the associated solutions of the discrete problem \eqref{eq: Nlink_bis_energ} on $[0,T]$ (Theorem \ref{thm: discrete well-posedness}). The corresponding interpolants $r^h,m^{h,z}, n^h$ are defined from (\ref{eq: interp_r_1} -- \ref{eq: interp_n}).

Assume that
\begin{equation}
r^h(0,\cdot) \xrightarrow[N \rightarrow + \infty]{} r^0 \text{ in } L^2(0,L), \label{eq: cvCI}
\end{equation}
and that there exists $C^0>0$, independent of $h$, such that, for all $N \in \mathbb{N}^*$,
\begin{equation}
\frac E 2 \sum\limits_{i=1}^{N-1} h \left(\frac{\theta_{i+1}(0)-\theta_i(0)}{h}\right)^2 = C^0_h \leq C^0 .
\label{eq: hyp_theta0}
\end{equation}

Then, there exist $r\in H^1(Q_T)$, $m\in  L^2(0,T; H^1(0,L))$ and $n \in L^2(0,T; H^1(0,L))$  such that, up to extraction of a subsequence when $h\to 0$, we have 
    \begin{align*}
        &r^h \rightharpoonup r \mbox{ weakly in } H^1(Q_T),\\
      &  m^{h,z} e_z \rightharpoonup m \mbox{ weakly in } L^2(0,T; H^1(0,L)),\\
      &  n^h \rightharpoonup n \mbox{ weakly in } L^2(0,T; H^1(0,L)).
    \end{align*}
    Moreover, $r \in L^2(0,T;H^2(0,L))$ and $(r,m,n)$ satisfy System \eqref{eq: continuous} for a.e. $(s,t)\in Q_T$ with initial condition $r^0$.
\end{theorem}

As we shall prove, for any initial data $r^0 \in C^2(0,L)$, it is possible to construct discrete initial conditions $\Xt^0$ satisfying hypotheses \eqref{eq: cvCI} and \eqref{eq: hyp_theta0}. As a consequence, we obtain the following existence result.

\begin{corollary}[Existence of solutions to the continuous system] \label{corollary}
   Given $r^0 \in C^2(0,L)$, there exists at least one solution $r\in H^1(Q_T)\cap L^2(0,T; H^2(0,L))$, $m\in L^2(0,T; H^1(0,L))$ and $n \in L^2(0,T; H^1(0,L))$ to System \eqref{eq: continuous} with initial condition $r^0$.
\end{corollary}

Despite both establishing existence of solutions for the elastohydrodynamics equation \eqref{eq: continuous}, Corollary \ref{corollary} and the result of Mori and Ohm \cite[Theorem 1.1]{mori_well-posedness_2023}, slightly differ. 
Indeed, Mori and Ohm prove local existence (and global existence for small enough initial data) and uniqueness for $\kappa$ (which stands for $\theta_s$ in their work) in $C^0([0,T];L^2(0,L))\cap C^0((0,T];H^1(0,L))$. On the other hand, Corollary \ref{corollary} establishes global existence of solutions (for any suitably regular initial data) such that $\theta_s=m^z/E$ belongs to $L^2(0,T;H^1(0,L))$.

Of course, an extension of \cite[Theorem 1.1]{mori_well-posedness_2023} to a space compatible with Corollary \ref{corollary} would allow to combine both results to provide global existence and uniqueness of solutions for any initial data. Recent developments have been obtained by Ohm \cite{ohm2025}, studying a three-dimensional version of the filament dynamics including shear deformation (Kirchhoff rod), and casting the dynamics in terms of $r_{ss}$ instead of $\kappa$. This paves the way to show that  \eqref{eq: continuous} is well posed for $r\in L^2(0,T;H^2(0,L))$. This would confirm the uniqueness of the solutions of Corollary \ref{corollary}. Consequently, the convergences obtained in Theorem~\ref{thm: convergence} would not hold only for subsequences but for the whole family $(r^h, m^{h,z}, n^h)$.

\section{Energy dissipation}\label{sec: energy}

The main idea to prove Theorem \ref{thm: discrete well-posedness} consists in ensuring that all the variables of System \eqref{eq: Nlink_bis_energ} remain bounded uniformly in time. For Theorem \ref{thm: convergence}, estimates on the interpolates \eqref{eq: interp_r_1}-\eqref{eq: interp_r_2}, uniform in $h$ and in well-chosen functional spaces, are the cornerstone of the proof. As a matter of fact, for both proofs, most of these key estimates derive from a single formula, that characterizes the dissipation of energy in solutions of System \eqref{eq: Nlink_bis_energ}. This formula is established in the following proposition.

\begin{proposition}[Energy dissipation]
\label{thm: energy}
Let $(r_i)_{1\leq i \leq N+1}$, $(\theta_i)_{1\leq i \leq N}$, $(m_i)_{1\leq i \leq N+1}$, $(n_i)_{1\leq i \leq N+1}$ be a solution to system \eqref{eq: Nlink_bis_energ}.
Then, the following identity holds for all $t \in \mathbb{R}_+$:
\begin{equation}
     \frac{1}{2}\frac{\d}{\d t}\left[E \sum\limits_{i=1}^{N-1} h \left(\frac{\theta_{i+1}-\theta_i}{h}\right)^2\right] + \sum\limits_{i=1}^N \frac{h^3}{12}c_{\perp} \dot \theta_i^2 - \sum\limits_{i=1}^N h\, \C(\theta_i)\,\dot r_{i+\frac12} \cdot \dot r_{i+\frac12} = 0\,. 
    \label{eq: energie}
\end{equation}
\end{proposition}

\begin{proof}
Multiplying the force balance equation \eqref{eq: Force balance} by the velocity $\dot r_{i+\frac12}$, and summing over $i=1,\dots,N$ leads to
\begin{align*}
    h\sum_{i=1}^{N} \C(\theta_i)\,\dot r_{i+\frac12}  \cdot \dot r_{i+\frac12} &= - \sum \limits_{i=1}^{N} \dot r_{i+\frac12} \cdot (n_{i+1}-n_i)\d s\\
    & = \sum\limits_{i=1}^N (\dot r_{i+1} - \dot r_i) \cdot \frac{n_{i+1}+n_i}{2},
\end{align*}
using a summation by parts and the boundary conditions $n_1 = n_{N+1}=0$.

But, since $r_{i+1}-r_i= h \,\eit$, we have $\dot r_{i+1}-\dot r_i= h\, \dot \theta_i \ein$. This, together with  the moment balance equation \eqref{eq: Nlink_bis_energ}, and the boundary conditions $m_1=m_{N+1}=0$ gives
\begin{align*}
    h\sum_{i=1}^{N} \C(\theta_i)\,\dot r_{i+\frac12}  \cdot \dot r_{i+\frac12} & = h \sum\limits_{i=1}^N \dot \theta_i e_{i, \perp} \cdot \frac{n_{i+1}+n_i}{2}\\
    & = h \sum\limits_i  \dot \theta_i e_z \cdot \eit \times \frac{n_{i+1}+n_i}{2} \\
     &= \sum\limits_{i=1}^N \dot \theta_i e_z \cdot \left[ m_i-m_{i+1}+\frac{h^3}{12}c_{\perp}\dot \theta_i e_z\right]\\
     & = \sum\limits_{i=1}^N \frac{h^3}{12}c_{\perp}\dot \theta_i^2 + \sum\limits_{i=1}^N \dot \theta_i e_z\cdot (m_i-m_{i+1})\\
     & = \sum\limits_{i=1}^N \frac{h^3}{12}c_{\perp}\dot \theta_i^2 + \sum\limits_{i=2}^N (\dot \theta_{i} - \dot \theta_{i-1})e_z\cdot m_i.
\end{align*}
 Now, observing that we have $m_{i} \cdot e_z =  \frac{E}{h}(\theta_{i} - \theta_{i-1})$ for $i=2,\cdots,N$, we deduce
\begin{equation*}
     \frac{E}{h} \sum\limits_{i=1}^{N-1} (\dot \theta_{i}-\dot\theta_{i-1}) (\theta_{i}-\theta_{i-1}) + \sum\limits_{i=1}^N \frac{h^3}{12}c_{\perp} \dot \theta_i^2 -h\sum_{i=1}^{N} \C(\theta_i)\,\dot r_{i+\frac12}  \cdot \dot r_{i+\frac12} = 0 
\end{equation*}
which directly yields \eqref{eq: energie}.

\end{proof}

\begin{remark}
Recalling from \eqref{eq:G} that $\dot r_{i+\frac12}$ depends linearly on $\dot \Xt = \left( \dot \theta_1, \, \dots, \, \dot \theta_N, \, \dot r_1^x, \dot r_1^y \right)^\intercal$, the energy identity \eqref{eq: energie} can also be rewritten in matrix form, 
\begin{equation}
\frac{1}{2}\frac{\d}{\d t}\left[E \sum\limits_{i=1}^{N-1} h \left(\frac{\theta_{i+1}-\theta_i}{h}\right)^2\right] + \dot \Xt^\intercal M(\Xt)\dot \Xt =0 .
\label{eq: matrix_energy}
\end{equation}
In \eqref{eq: matrix_energy}, $M(\Xt) = M_1+M_2(\Xt)$ with
\begin{equation*}
M_1 = \frac{h^3}{12}\cn \begin{pmatrix}
1 &  &  &  & \\
 & \ddots  &  &  & \\
 &  & 1 &  & \\
 &  &  & 0 & \\
 &  &  &  & 0
\end{pmatrix}\,,\,\,M_2 = -h G(\Xt)^\intercal \mathbb{C} (\Xt)G(\Xt),
    \label{eq: M1}
\end{equation*}
where $G(\Xt)$ and $\mathbb{C} (\Xt)$ are defined in equations \eqref{eq:G} and \eqref{eq:CC}.
\end{remark}

\begin{remark} Notice that the various discrete integrations by parts in the proof of Proposition \ref{thm: energy} remain true when one considers pinned \eqref{eq:BC proximal pinned disc} or clamped \eqref{eq: BC proximal clamped disc} boundary conditions, so \eqref{eq: int_energie} still holds for those cases.
\end{remark}

\section{Proofs} \label{sec: proofs}

\subsection{Proof of Theorem \ref{thm: discrete well-posedness}}\label{sec: proof existence}

The proof of the existence of a global solution to the N-link model is structured as follows: first, we show the local-in-time well-posedness of the system \eqref{eq: Nlink_bis_energ} or equivalently~\eqref{eq: matrixAXF} for free boundary conditions at both ends. Then, we deduce the global existence.\\

\noindent\textbf{Local-in-time existence and uniqueness.}
Let $N \in \mathbb{N}^*$, and let us prove that, for a set of initial conditions $\Xt^0=(\theta_1^0, \, \dots,\, \theta_N^0,\, r_1^{x,0},\, r_1^{y,0}) \in \mathbb{R}^{N+2}$, System~\eqref{eq: matrixAXF} admits a unique solution $(\mathbf{N},\Xt, \mathbf{M})$  in $C^0([0,T])^{2N}\times C^1([0,T])^{N+2} \times C^0([0,T])^{N}$ for $T>0$ sufficiently small.
We recall the matrix formulation~\eqref{eq: matrixAXF} of the $N$-link system from equation \eqref{eq: matrixAXF}: 
\[
\left(
\begin{array}{ccc}
A_{11} & A_{12}(\Xt) & 0_{2N,N}\\
A_{21}(\Xt) & A_{22} & A_{23}\\
0_{N,2N} & 0_{N,N+2} & I_N
\end{array}
\right)
\left(
\begin{array}{c}
\mathbf{N}\\ \dot \Xt\\ \mathbf{M}
\end{array}
\right)= 
\left(
\begin{array}{c}
0_{2N}\\ 0_{N+2}\\ F_3(\Xt)
\end{array}
\right)\,.
\]

Note that the system in this form is an differential-algebraic system, combining differential equations on $\mathbf{X}$ and algebraic equations on $\mathbf{N}$ and $\mathbf{M}$. It is suitable to recast it as a differential system, which in turn allows use of the Cauchy-Lipschitz theorem to establish the existence of solutions. To do so, note that the matrix 
$\left(
\begin{array}{cc}
A_{11} & 0_{2N,N}\\0_{N,2N} & I_N
\end{array}
\right)$ is clearly invertible. Therefore, we can rewrite~\eqref{eq: matrixAXF} as
\begin{equation}
    B(\Xt)\dot \Xt=\tilde{F}(\Xt)
    \label{eq: BtildeF}
\end{equation}
with
\begin{equation}
B(\Xt) = A_{22} - (A_{21}(\Xt)\,A_{23}) 
\begin{pmatrix}
A_{11} & 0_{2N,N}\\0_{N,2N} & I_N
\end{pmatrix}^{-1}
\begin{pmatrix}
A_{12}(\Xt)\\0_{N,N+2}
\end{pmatrix} = A_{22}-A_{21}(\Xt)A_{11}^{-1}A_{12}(\Xt)
\label{eq: definition B}
\end{equation}
and
\begin{equation*}
    \tilde F(\Xt) = F_2 - A_{23}F_3(\Xt)-A_{21}(\Xt)A_{11}^{-1}F_1 = -A_{23}F_3(\Xt)\,.
\end{equation*}
Both $B$ and $\tilde{F}$ are Lipschitz continuous in $\Xt$ since $A_{12},\,A_{21}$ and $F_3$ depend on $\Xt$ in Lipschitzian way. \\

Now, it remains to show that, for all $\mathbf{X}=(\theta_1, \, \dots,\, \theta_N,\, r_1^{x},\, r_1^{y})  \in \mathbb{R}^{N+2}$, $B(\mathbf{X})$ is invertible or, since it is a square matrix, one-to-one.  From \eqref{eq: definition B}, one can see that it is sufficient to prove that $\mathbf{A}(\Xt) = \begin{pmatrix} A_{11} & A_{12}(\Xt) \\ A_{21}(\Xt) & A_{22} \end{pmatrix}$ is one-to-one (it is in fact equivalent). Indeed, suppose that $\mathbf{A}(\Xt)$ is one-to-one and let $\mathbf{W}\in \mathbb{R}^{N+2}$ be such that $B(\mathbf{X}) \mathbf{W}=0$. By direct computation, setting $\mathbf{V}=-A_{11}^{-1} A_{12}(\Xt)\mathbf{W} $, we obtain $\mathbf{A}(\Xt) (\mathbf{V},\mathbf{W})^\intercal =0$, which implies $(\mathbf{V},\mathbf{W})=0$. In particular, $\mathbf{W}=0$ which proves that $\mathbf{B}(\Xt)$ is one-to-one.

So, let us now prove that $\mathbf{A}(\Xt)$ is indeed one-to-one. We take $\mathbf{V} = (V_i)_i \in \mathbb{R}^{2N}$ and $\mathbf{W} = ( W_i)_i \in \mathbb{R}^{N+2}$ such that 
    \begin{equation*}
       \mathbf{A}(\Xt)\begin{pmatrix}
            \mathbf{V} \\ \mathbf{W}
        \end{pmatrix} = 0_{3N+2}\,
        \label{eq: a11 a22}
    \end{equation*}
that we rewrite as
 \begin{align}
            A_{11}\mathbf{V}+A_{12}(\Xt)\mathbf{W} & =0_{2N}\,, 
            \label{eq: systNX-1}
            \\
            A_{21}(\Xt)\mathbf{V} +A_{22}\mathbf{W}& =0_{N+2}\,,
        \label{eq: systNX-2}
    \end{align} 
and prove that, necessarily, $\mathbf{V} = 0_{2N}$ and $\mathbf{W}= 0_{N+2}$. 
First, notice that the last two lines in \eqref{eq: systNX-2} imply $V_1 = V_2 = 0$. Then we compute
\begin{align*}
    0 & = \mathbf{W}^\intercal(A_{21}(\Xt)\mathbf{V} +A_{22}\mathbf{W})\\
    & = \sum\limits_{i=1}^N h\, W_i \begin{pmatrix}  -\sin\theta_i \\ \cos \theta_i \end{pmatrix} \cdot \frac{V_{i+1} + V_i}{2} +   \frac{h}{2}\,W_N\begin{pmatrix}  -\sin \theta_N \\ \cos \theta_N \end{pmatrix} \cdot V_N - \sum\limits_{i=1}^N \frac{h^3}{12}c_{\perp}W_i^2\,.
  \end{align*}
 
 Using $G(\Xt)$ as introduced in equation \eqref{eq:G}, and the fact that $V_1 = V_2 = 0$, one has, by summation by parts,
  \begin{align*}
    0 &  = \frac{h}{2}  \sum\limits_{i=2}^N  V_i \cdot\left(W_{i-1} \begin{pmatrix}  -\sin \theta_{i-1} \\ \cos \theta_{i-1} \end{pmatrix} + W_i \begin{pmatrix}  -\sin \theta_i \\ \cos \theta_i \end{pmatrix}\right)- \sum\limits_{i=1}^N \frac{h^3}{12}c_{\perp}W_i^2\\
    & = \sum\limits_{i=2}^N  V_i \cdot ((G(\Xt)\mathbf{W})_{i}-(G(\Xt)\mathbf{W})_{i-1})- \sum\limits_{i=1}^N \frac{h^3}{12}c_{\perp}W_i^2\\
    & = - \sum\limits_{i=1}^N (V_{i+1} - V_i)\cdot(G(\Xt)\mathbf{W})_i- \sum\limits_{i=1}^N \frac{h^3}{12}c_{\perp}W_i^2.
\end{align*}
Using \eqref{eq: systNX-1}, we also have that for $1 \leq i \leq N$, $V_{i+1}-V_i = -h (\mathbb{C}(\Xt)G(\Xt)\mathbf{W})_i$ and we recover
\begin{equation}
    0 =  \sum\limits_{i=1}^N h (\mathbb{C}(\Xt)G(\Xt)\mathbf{W})_i\cdot(G(\Xt)\mathbf{W})_i- \sum\limits_{i=1}^N \frac{h^3}{12}c_{\perp}W_i^2.
    \label{eq: def pos}
\end{equation}
The matrix $\mathbb{C}(\Xt)$ is block diagonal with negative definite blocks, so both terms in \eqref{eq: def pos} are negative, which then yields $W_i = 0$ and $(G(\Xt)\mathbf{W})_i = 0$ for all $i\in \{1,\cdots,N\}$. Moreover, from \eqref{eq:G}, we deduce that
$$
0=(G(\Xt)\mathbf{W})_1=\left(\begin{array}{c}W_{N+1}\\W_{N+2}\end{array}\right) + \frac{h}{2}\begin{pmatrix} -\sin \theta_1 \\ \cos  \theta_1 \end{pmatrix} V_1 = \left(\begin{array}{c}W_{N+1}\\W_{N+2}\end{array}\right)\,,
$$
which finishes to prove that $\mathbf{W}=0_{N+2}$.
Finally, using \eqref{eq: systNX-1}, we deduce that $\mathbf{V}=0_{2N}$, which means that $\mathbf{A}(\Xt)$, and hence $B(\Xt)$, are both invertible.

We can now apply Cauchy-Lipschitz theorem to equation \eqref{eq: BtildeF}, and conclude that for any initial condition, System \eqref{eq: Nlink_bis_energ}  admits a unique $C^1$ solution $\Xt=(\theta_1,\, \dots,\, \theta_N, r_1)$ locally in time.\\    

\noindent\textbf{Global existence and uniqueness.}
The solution $\Xt=(\theta_1,\, \dots,\, \theta_N, r_1)$ to System \eqref{eq: Nlink_bis_energ} can be extended to $\mathbb{R}_+$ as long as it does not blow up in finite time, which is ensured by the following lemma.

\begin{lemma}[Bounds on $(r_i)_i$ and $(\theta_i)_i$] 
    For any $T>0$, assume that we can define all  $(r_i)_{1\leq i \leq N+1}$ and $(\theta_i)_{1\leq i \leq N}$ on $[0, T[$. Then, they stay bounded in $[0,T]$ in the following sense:
    \begin{eqnarray}
    \vert \theta_i(t) - \theta_i(0) \vert &\leq& C_1T^{1/2},
    \label{eq:estimtheta}
    \\
    \vert r_i(t) - r_i(0) \vert &\leq& C_2 T^{1/2},
    \label{eq:estimr}
    \end{eqnarray}
    for all $1 \leq i \leq N$, for all $t\in [0,T]$, and where $C_1, C_2>0$ are constants that do not depend on $T$. 
    \label{thm: bounds_rh}
\end{lemma}

\begin{proof} Let $T>0$ and assume that the solution $(r_i)_{1\leq i \leq N+1}$ and $(\theta_i)_{1\leq i \leq N}$ exists on $[0, T[$.
First of all, integrating  equation \eqref{eq: energie} over time $t\in [0,T]$, we obtain
\begin{equation}
\begin{array}{cc}
     \displaystyle \frac{1}{2}\left[E \sum\limits_{i=1}^{N-1} h \left(\frac{\theta_{i+1}(T)-\theta_i(T)}{h}\right)^2\right] + \int_0^T\sum\limits_{i=1}^N \frac{h^3}{12}c_{\perp} \dot \theta_i^2(t)\d t\hspace*{4cm}\\
     \displaystyle \hspace*{4cm}-  \int_0^T\sum\limits_{i=1}^N h\, \C(\theta_i)\,\dot r_{i+\frac12}(t) \cdot \dot r_{i+\frac12}(t)\d t = C^0_h, 
    \end{array}
    \label{eq: int_energie}
\end{equation}
where the terms on the left-hand side are all positive, and $C^0_h = \frac{1}{2}\left[E \sum\limits_{i=1}^{N-1} h \left(\frac{\theta_{i+1}(0)-\theta_i(0)}{h}\right)^2\right]>0$ is a constant that depends on the initial condition and on $h$, but not on $T$. 

To prove \eqref{eq:estimtheta}, we write
\begin{eqnarray*}
    \sum_{i=1}^{N }\vert \theta_i(t) - \theta_i(0) \vert &=& \sum_{i=1}^{N } \left| \int_0^t \dot \theta_i(u) \d u\right| \leq  \sum_{i=1}^{N } \int_0^T \left| \dot \theta_i(u)\right| \,\d u
    \\
    &\leq& \sum\limits_{i=1}^{N} T^{1/2} \left(\int_0^T  \vert \dot \theta_i(u)\vert^2\d u\right)^{1/2}
   \leq (NT)^{1/2} \left(\sum_{i=1}^{N} \int_0^T \left| \dot \theta_i(u)\right|^2 \d u \right)^\frac12 \\
   &\leq&  (NT)^{1/2} \left(\frac{12}{h^3 \cn}C_h^0\right)^\frac12,
\end{eqnarray*}
where the last inequality comes from~\eqref{eq: int_energie}.

It now remains to prove~\eqref{eq:estimr}. First, we write similarly
\begin{equation}
    \sum_{i=1}^{N } \vert r_i(t) - r_i(0) \vert = (NT)^{1/2} \left(\sum_{i=1}^{N} \int_0^T \left| \dot r_i(u)\right|^2 \d u \right)^\frac12.
\label{eq: ri minus r0}
\end{equation}

To bound the right-hand side of \eqref{eq: ri minus r0}, we proceed in two steps. On the one hand, the matrix $\C(\theta_i)$ is negative definite for all $(\theta_i)$, with eigenvalues $(-\ct,-\cn)$ independent of $i$, so that from~\eqref{eq: int_energie} we have
\begin{equation}
    \int_0^T \sum\limits_{i=1}^N h\,\vert \dot r_{i+1/2}(u) \vert^2 \d u\leq \tilde C_h^0,
    \label{eq: borne_ridemi}
\end{equation}
where $\tilde C_h^0 = C_h^0/\min (\ct,\cn)$.

On the other hand, using \eqref{eq: Ri}, one can write $\dot r_{i+1}- \dot r_i = h\, \dot \theta_i \ein$. Hence, from~\eqref{eq: int_energie} we obtain
 \begin{equation}
     \displaystyle \int_0^T\sum\limits_{i=1}^N h \,\vert \dot r_{i+1}(u)-\dot r_i(u) \vert^2 \d u= \displaystyle \int_0^T\sum\limits_{i=1}^N h\, \vert h \,\dot \theta_i(u) \ein \vert^2 \d u= \int_0^T \sum\limits_{i=1}^N h^3 \dot \theta_i^2(u) \d u \leq \frac{12}{\cn}C_h^0.
     \label{eq: diff_ri}
 \end{equation}
 
Combining equations \eqref{eq: borne_ridemi} and \eqref{eq: diff_ri} finally leads to 
\begin{equation}
\int_0^T \sum\limits_{i=1}^{N+1} h \vert \dot r_i(u)\vert^2 \d u \leq 4\tilde C_h^0 + \frac{12}{\cn}C_h^0.
\label{eq: borne_ridot}
\end{equation}
which concludes the proof.
\end{proof}

Lemma \ref{thm: bounds_rh} guarantees that the solution $(\theta_1,\, \dots,\, \theta_N, r_1)$ to system \eqref{eq: Nlink_bis_energ} can not blow up in finite time. The system \eqref{eq: Nlink_bis_energ} with initial conditions $(\theta_1^0, \, \dots,\, \theta_N^0, r_1^0) \in \mathbb{R}^{N+2}$ therefore admits a unique $C^1$ solution for all time, which concludes the proof of Theorem \ref{thm: discrete well-posedness}.\\

\begin{remark}
For pinned (resp. clamped) boundary conditions, the system and the proof can be adapted by removing  $(r_1^x,r_1^y)$ (resp. $(r_1^x,r_1^y,\theta_1$)) from the unknowns and by looking at a new matrix in~\eqref{eq: matrixAXF}, of size $4N \times 4N$ (resp. $(4N-1) \times (4N-1)$).
\end{remark}

\subsection{Proof of Theorem \ref{thm: convergence}}\label{sec: convergence}

In this section, we prove Theorem \ref{thm: convergence}: namely, that the discrete solution to the N-link model computed from Theorem \ref{thm: discrete well-posedness} converges towards the solution to the continuous model.

To begin, we assume that $r^0\in C^2(0,L)$ is given together with a representative $\theta^0\in C^1(0,L)$  such that $r^0_s = (\cos \theta^0, \, \sin \theta^0)^\intercal \in C^1(0,L)$.

Let $T>0$. For all $N \in \mathbb{N}^*$, a set of discrete initial conditions $\Xt^0=(\theta_1^0, \, \dots,\, \theta_N^0,\, r_1^{x,0},\, r_1^{y,0}) \in \mathbb{R}^{N+2}$ is given from which the solutions of the discrete problem \eqref{eq: Nlink_bis_energ} is computed on $[0,T]$ thanks to Theorem \ref{thm: discrete well-posedness}. The corresponding interpolants $r^h,m^{h,z}, n^h, \bar\theta^h, \theta^h$ and $\bar m^{h,z}$ are then defined from (\ref{eq: interp_r_1} -- \ref{eq: interp_r_2}).

Furthermore we assume that \eqref{eq: cvCI} and \eqref{eq: hyp_theta0} hold, namely
$$
r^h(0,\cdot) \xrightarrow[N \rightarrow + \infty]{} r^0 \text{ in } L^2(0,L), 
$$
$$
\frac E 2 \|\theta_s^h(0,\cdot)\|_{L^2(0,L)} = \frac E 2 \sum\limits_{i=1}^{N-1} h \left(\frac{\theta_{i+1}(0)-\theta_i(0)}{h}\right)^2 = C^0_h \leq C^0  
$$
where $C^0$ does not depend on $h$.

The proof is split into three propositions. In Proposition \ref{thm: global_bounds}, we bound the interpolates independently of $h$ in suitable function spaces. Then, Proposition \ref{thm: extraction} establishes the existence of a limit to each of these interpolates as $h$ goes to zero. Finally, Proposition \ref{thm: formefaible} shows that this limit is a solution of System \eqref{eq: continuous} in a weak sense and the proof is concluded proving that the limit satisfy System \eqref{eq: continuous} almost everywhere in $Q_T$, with the initial condition $r^0$. \\

Let $BV(0,L)$ be defined as the space of functions of bounded variation on $[0,L]$, equipped with the norm $\Vert u \Vert_{BV(0,L)} = \| u\|_{L^1(0,L)} + TV_0^L(u)$, for $u \in BV(0,L)$, where $TV_0^L(u)$ is the total variation of $u$ on $(0,L)$ \cite{ambrosio2000functions}.

\begin{proposition}[Bounds on interpolates]
    If $(\theta^h_s(0,\cdot))_h$ is bounded uniformly $h$ in $L^2(0,L)$, the interpolates defined in equations (\ref{eq: interp_r_1}-\ref{eq: interp_r_2}) satisfy the following bounds, uniformly in h.
    \begin{enumerate}
    \item $(r^h)_h$ is bounded in $H^1(Q_T)$;
    \item $(\dot r^h_s)_h$ is bounded in $L^2(0,T; H^{-1}(0,L))$;
    \item $(r^h_s)_h$ is bounded in $L^2(0,T; BV(0,L))$;
    \item $(n^h)_h$ is bounded in $L^2(0,T; H^1(0,L))$;
    \item $(m^{h,z})_h$ is bounded in $L^2(0,T; H^1(0,L))$;
    \item $(h \, \dot{\bar{\theta}}^h)_h$ is bounded in $L^2(Q_T)$;
    \item $(\theta^h_s)_h$ is bounded in $L^2(Q_T)$.
    \end{enumerate}
    \label{thm: global_bounds}
\end{proposition}

\begin{proof}
\noindent \textbf{Point 1. Bounding $(r^h)_h$ in $H^1(Q_T)$.} 
Since $r^h$ is the piecewise linear interpolate of the $(r_i)_{1\leq i\leq N+1}$, it comes that
\begin{equation*}
    r^h_s(t,s) = \displaystyle 
 \sum\limits_{i=1}^{N} \frac{r_{i+1}(t)-r_i(t)}{h} \mathbbm{1}_{L_i}(s)\,.
 \label{eq: rh bound}
\end{equation*} 
Then, the inextensibility condition \eqref{eq: inext} implies that $\vert r^h_s\vert = \vert \frac{r_{i+1}-r_i}{h}\vert =1$ which ensures that $(r^h_s)_h$ is bounded in $L^2(Q_T)$. 

Then, we bound $\dot r^h $ in $L^2(Q_T)$. Direct calculations show that 
 \begin{eqnarray*}
     \Vert \dot r^h \Vert_{L^2(Q_T)}^2  &=& \int_0^T\sum\limits_{i=1}^N \int_{L_i} \left| \dot r_i(t) + s \frac{\dot r_{i+1}(t)-\dot r_i(t)}{h}\right|^2 \d s\,\d t\nonumber\\
     &\leq& \displaystyle 2\int_0^T\sum\limits_{i=1}^N \int_{L_i} \left| \dot r_i(t)\right|^2 \d s\,\d t+ 2 \int_0^T\sum\limits_{i=1}^N \int_{L_i} s^2 \frac{(\dot r_{i+1}(t)-\dot r_i(t))^2}{h^2}\d s\,\d t\nonumber\\
      &\leq& \displaystyle2 \int_0^T\sum\limits_{i=1}^N h \,\left| \dot r_i(t) \right|^2\,\d t\, + \frac{2}{3}\int_0^T\sum\limits_{i=1}^N  h^3 \dot \theta_i^2(t)\,\d t\,.
 \label{eq: bound_rdoth}
 \end{eqnarray*}
This gives the bound using \eqref{eq: borne_ridot} for the first term and (\ref{eq: int_energie},\ref{eq: hyp_theta0}) for the second one. 

It now remains to bound $r^h$ in $L^2(Q_T)$. 
To do so, we write $r^h(t,s) = \bar r^h (t)+ \delta r^h(t,s)$, with $ \bar r^h(t) = \frac{1}{L} \int_0^L r^h(t, s)\d s$. Then, the Poincaré-Wirtinger inequality yields
\begin{equation*}
    \Vert \delta r^h \Vert_{L^2(Q_T)} =  \Vert r^h - \bar r^h \Vert_{L^2(Q_T)} \leq C \Vert r^h_s \Vert_{L^2(Q_T)},
\end{equation*}
which is bounded uniformly in $h$. Then, to bound $\bar r^h$ in $L^2(Q_T)$, we write
\begin{equation*}
    \vert \bar r^h(t) - \bar r^h(0) \vert \leq \int_0^T \vert \dot{ \bar{ r}}^h \vert \leq T^{1/2} \left( \int_0^T \vert \dot{\bar{r}}^h\vert^2\right)^{1/2},
\end{equation*}
and using the $L^2(Q_T)$ orthogonal decomposition $\dot r^h  = \dot{\bar{r}}^h + \dot {\delta r}^h$, we obtain
\begin{equation*}
    \vert \bar r^h(t) - \bar r^h(0) \vert \leq \left(\frac T L \right)^{1/2} \|\dot r^h\|_{L^2(Q_T)} \leq C,
\end{equation*}
which concludes the proof.\\

\noindent\textbf{Point 2. Bounding $(\dot r^h_s)_h$ in $L^2(0,T; H^{-1}(0,L))$.} 
Since $(\dot r^h)_h$ is bounded in $L^2(Q_T)$ (Point 1. above), we immediately deduce that $(\dot r^h_s)_h$ is bounded in $L^2(0,T;H^{-1}(0,L))$ because, for all $u \in L^2(0,L)$, $\Vert u_s \Vert_{H^{-1}(0,L)} \leq \Vert u \Vert_{L^2(0,L)}$. \\

\noindent \textbf{Point 3. Bounding $(r^h_s)_h$ in $L^2(0,T; BV(0,L))$.} 
Since $r_s^h$ is piecewise constant, if we define $r^h_{s,i}$ as the value of $r_s^h$ on $[(i-1)h,ih]$ we have
\begin{equation*}
    TV_0^L(r^h_s(t, \, \cdot)) = \sum\limits_{i=1}^{N-1} \vert r^h_{s,i+1}(t)  - r^h_{s,i}(t) \vert.
\end{equation*}  
Recalling that $\eit(t) = \frac{r_{i+1}(t)-r_i(t)}{h} = r^h_{s,i}$ and $|r^h_s(t,s)|=1$ for any $(t,s)\in Q_T$,  we get 
\begin{equation*}
    \Vert r^h_s(t, \, \cdot)\Vert_{BV(0,L)}  = \|r^h_s(t,\cdot)\|_{L^1(0,L)} + TV_0^L(r^h_s(t, \, \cdot)) = L + \sum\limits_{i=1}^{N-1} \vert e_{i+1, \parallel}(t) - \eit(t)\vert.
\end{equation*}
Then, using Cauchy-Schwarz inequality and the fact that $\theta \to (\cos\theta, \sin\theta)^\intercal$ is $1$-Lipschitz, it comes that
\begin{eqnarray*}
     \Vert r^h_s(t, \, \cdot)\Vert_{BV} \leq L + \sum\limits_{i=1}^{N-1} \vert \theta_{i+1}(t)-\theta_i(t)\vert &\leq& L + \sqrt{N-1} \left( \sum\limits_{i=1}^{N-1} \vert \theta_{i+1}(t)-\theta_i(t) \vert^2\right)^{1/2}\\
     &\leq& L + \sqrt L \left( \sum\limits_{i=1}^{N-1} h \frac{\vert \theta_{i+1}(t)-\theta_i(t) \vert^2}{h^2}\right)^{1/2}
\end{eqnarray*}
which is bounded from (\ref{eq: int_energie},\ref{eq: hyp_theta0}) and from which we deduce the $L^2(0,T; BV(0,L))$ bound.\\

\noindent \textbf{Point 4. Bounding $(n^h)_h$ in $L^2(0,T; H^1(0,L))$.} First, using the boundary condition $n^h(t,L) =0$, it is sufficient to prove that $(n^h_s)_h$ is bounded in $L^2(Q_T)$.
But, using the force balance of \eqref{eq: Nlink_bis_energ}, we notice that 
\begin{eqnarray*}
    \Vert n^h_s(t,\, \cdot) \Vert_{L^2(0,L)}^2 &=& \sum\limits_{i=1}^N h\, \left \vert \frac{n_{i+1}(t)-n_i(t)}{h} \right \vert^2\\
    &=&\sum\limits_{i=1}^N h\,\left \vert \C(\theta_i(t)) \dot r_{i+1/2}(t)\right \vert^2.
\end{eqnarray*}

Integrating over $[0,T]$, we get
\begin{equation*}
    \Vert n^h_s \Vert_{L^2(Q_T)}^2=  \int_0^T \sum \limits_{i=1}^N  h \left \vert \C(\theta_i(t)) \dot r_{i+1/2}(t)\right \vert^2\mathrm{d} t.
\end{equation*}
Using again (\ref{eq: int_energie},\ref{eq: hyp_theta0}) and the fact that $\C(\theta)$ is bounded, we obtain that $(n^h_s)_h$ is bounded in $L^2(Q_T)$ uniformly in $h$. \\

\noindent \textbf{Point 5. Bounding $(m^{h,z})_h$ in $L^2(0,T; H^1(0,L))$.} As for $n^h$ before, since, for all $t>0$, $m^{h,z}(t,L)=0$, it is sufficient to prove that $(m^{h,z}_s)_h$ is bounded in $L^2(Q_T)$. Let us write
\begin{equation}
    \Vert m^{h,z}_s(t, \, \cdot) \Vert_{L^2(0,L)}^2 = \sum\limits_{i=1}^N h \left| \frac{m_{i+1}^z(t)-m_i^z(t)}{h}\right|^2= \sum\limits_{i=1}^N h \left| \frac{m_{i+1}(t)-m_i(t)}{h}\right|^2.
    \label{eq: mhs}
\end{equation}
Then, using the moment balance from the system of equations \eqref{eq: Nlink_bis_energ}, we also have that
\begin{equation}
    \left| \frac{m_{i+1}(t)-m_i(t)}{h}\right| \leq \left| \frac{n_{i+1}(t)+n_i(t)}{2}\right| + \frac{h^2}{12}\cn \vert \dot \theta_i(t)\vert.
    \label{eq: ITmhs}
\end{equation}
Combining equations \eqref{eq: mhs} and \eqref{eq: ITmhs} leads to
\begin{equation}
    \Vert m^{h,z}_s(t, \, \cdot) \Vert_{L^2(0,L)}^2 \leq 2\sum\limits_{i=1}^N h \left| \frac{n_{i+1}(t)+n_i(t)}{2}\right|^2 + 2\sum\limits_{i=1}^N h  \frac{h^4}{144}\cn^2 \vert \dot \theta_i(t)\vert^2.
    \label{eq mhs+Itmhs}
\end{equation}
Moreover, one can notice that 
\begin{align}
    \Vert n^h(t, \, \cdot)\Vert_{L^2(0,L)}^2 &= \sum\limits_{i=1}^N \int_{-h/2}^{h/2} \left| \frac{n_i(t) + n_{i+1}(t)}{2}+\frac{2s}{h} \frac{n_{i+1}(t)-n_i(t)}{2}\right|^2\d s\nonumber\\
    &\geq \sum\limits_{i=1}^N h \left|\frac{n_i(t)+n_{i+1}(t)}{2}\right|^2.
    \label{eq: norm_nh}
\end{align}
Using equation \eqref{eq: norm_nh} into equation \eqref{eq mhs+Itmhs} and integrating over time then gives
\begin{equation*}
    \Vert m^{h,z}_s \Vert_{L^2(Q_T)}^2
    \leq 2 \int_0^T\Vert n^h(t, \, \cdot)\Vert_{L^2(0,L)}^2\d t + \int_0^T\frac{2 h^2}{144}\cn^2 \sum\limits_{i=1}^N h^3 \vert \dot \theta_i (t)\vert^2\d t,
\end{equation*}
which is bounded uniformly in $h$, by virtue of Point 4. and Equations (\ref{eq: int_energie},\ref{eq: hyp_theta0}).\\

\noindent\textbf{Point 6. Bounding $(h \, \dot{\bar{\theta}}^h)_h$ in $L^2(Q_T)$.} 
We have 
\begin{equation*}
    \int_0^T \sum\limits_{i=1}^N \int_{L_i} (h\,  \dot{\bar{\theta}}^h(t,s))^2\d s\,\d t = \int_0^T\sum\limits_{i=1}^N h^3\,  \dot \theta_i^2(t)\,\d t,
    \label{eq: bound_deltasdottheta}
\end{equation*}
which is again bounded from (\ref{eq: int_energie},\ref{eq: hyp_theta0}).\\

\noindent\textbf{Point 7. Bounding $(\theta^h_s)_h$ in $L^2(Q_T)$.}
 
Writing $\theta^h_s(t,s) = \sum\limits_{i=1}^{N-1} \frac{\theta_{i+1}-\theta_i}{h}\mathbbm{1}_{L_{i+1}}$ we have
\begin{equation*}
    \Vert \theta^h_s \Vert_{L^2(Q_T)}^2  =  \int_0^T \sum\limits_{i=1}^{N-1} h \left|\frac{\theta_{i+1}(t)-\theta_i(t)}{h} \right|^2 \d t
\end{equation*}
which is bounded from (\ref{eq: int_energie},\ref{eq: hyp_theta0}).
\end{proof}

\medskip

From the previous estimates, we can now establish the convergence of the interpolates.

\begin{proposition}[Convergent subsequences] \label{thm: extraction}
There exist $r\in H^1(Q_T)\, \cap\, L^2(0,T; H^2(0,L))$, $n\in L^2(0,T; H^1(0,L))$, $m^z\in L^2(0,T; H^1(0,L))$ and $\alpha \in L^2(Q_T)$ such that up to the extraction of a subsequence, as $h\to 0$:
\begin{enumerate}
    \item $(r^h)_h$ converges to $r$ weakly in $H^1(Q_T)$ and strongly in $L^2(Q_T)$;
    \item $(r^h_s)_h$ strongly converges to $r_s$ in $L^2(0,T; L^p(0,L))$ for all $1 \leq p < \infty$;  
    \item $(\hat r^h)_h$ converges to $r$ strongly in $L^2(0,T; H^1(0,L))$ and weakly in $L^2(0,T; H^2(0,L))$;
    \item $(n^h)_h$ weakly converges to $n$ in $L^2(0,T; H^1(0,L))$;
    \item $(m^{h,z})_h$ weakly converges to $m^z$ in $L^2(0,T; H^1(0,L))$;
    \item $(\bar m^{h,z})_h$ weakly converges to $m^z$ in $L^2(0,T; L^p(0,L))$ for all $1 \leq p < \infty$;
    \item $(h^2 \, \dot{\bar{\theta}}^h)_h$ strongly converges to $0$ in $L^2(Q_T)$;
    \item $(\theta^h_s)_h$ weakly converges to $\alpha$ in $L^2(Q_T)$. 
\end{enumerate}
\end{proposition}

\begin{proof}
\textbf{Points 1., 4., 5., 7. and 8.} These points immediately follow from the bounds 1., 4., 5., 6. and 7. in Proposition \ref{thm: global_bounds}. Notice that Point 1. is obtained using also Rellich-Kondrachov theorem.\\

\noindent \textbf{Point 2. Convergence of $(r^h_s)_h$ in $L^2(0,T; L^p(0,L))$ for all $1 \leq p < \infty$.}
We know from Proposition \ref{thm: global_bounds}, that $(\dot r^h_s)_h$ is bounded in $L^2(0,T; H^{-1}(0,L))$ and $(r^h_s)_h$ is bounded in $L^2(0,T; BV(0,L))$. Let us recall that for all $1\leq p < \infty$, $BV(0,L)$ is compactly embedded in $L^p(0,L)$ \cite[Corollary 3.49]{ambrosio2000functions}. Then, from Aubin-Lions-Simon theorem \cite{aubin1963theoreme, boyer2012mathematical}, we  deduce that $(r^h_s)_h$ is relatively compact in $L^2(0,T; L^p(0,L))$.\\

\noindent \textbf{Point 3. Convergence of $(\hat r^h_s)_h$ in $L^2(0,T; H^1(0,L))$ and $L^2(0,T; H^2(0,L))$.}
Using equation \eqref{eq: Nlink_bis_energ}, we have
\begin{align*}
    \Vert \bar \theta^h - \theta^h\Vert_{L^2(Q_T)}^2 &= \int_0^T \sum\limits_{i=2}^{N} \int_0^{h} (\theta_i - \theta_{i-1})^2 \left(1- \frac{s}{h}\right)^2\d s\d t\\
    & = \int_0^T \sum\limits_{i=2}^{N}  \frac{h^3}{3E^2}  {m_{i+1}^z}^2\d t \\
    & \leq \frac{h^2}{3E^2}\Vert \bar m^{h,z} \Vert^2_{L^2(Q_T)}.
\end{align*}
But $(\bar m^{h,z})_h$ is bounded in $L^2(Q_T)$. Indeed, $(m^{h,z})_h$ is bounded in $L^2(Q_T)$ from Point 5. in Proposition \ref{thm: global_bounds}, and computing the difference   
\begin{align}
      \displaystyle \int_0^T \Vert \bar{m}^{h,z}(t,\cdot)-m^{h,z}(t,\cdot) \Vert_{L^2(0,L)}^2\d t &  \displaystyle = \int_0^T\sum\limits_{i=1}^N \int_0^{h} (m_i^z(t)(1-s/h)+m_{i+1}^z(t)s/h - m_{i}^z(t))^2\d s\,\d t  \nonumber\\[6pt]
         &\displaystyle = \int_0^T \sum\limits_{i=1}^N \int_0^{h} 1/E^2 (m_i^z(t)-m_{i+1}^z(t))^2(s/h)^2\d s\,\d t\nonumber\\[6pt]
         & \displaystyle \leq \frac{h}{3} \Vert m^{h,z} \Vert_{L^2(0,T;H^1(0,L))}^2,
    \label{eq: bornemi}
\end{align}
we obtain the claimed bound on $\bar m^{h,z}$. Hence, $\Vert \bar \theta^h - \theta^h\Vert_{L^2(Q_T)}$ goes to zero when $h \to 0$. Then, this also means that
\begin{equation*}
     \lim_{h \to 0} \Vert \hat r^h_s -  r^h_s \Vert_{L^2(Q_T)} = \lim_{h \to 0} \Vert (\cos \theta^h,\, \sin \theta^h)^\intercal - (\cos \bar \theta^h, \, \sin \bar \theta^h)^\intercal \Vert_{L^2(Q_T)} = 0.
\end{equation*}
Since, from Point 2., $(r^h_s)_h$ strongly converges towards $r_s$ in $L^2(Q_T)$, we also have that $(\hat r^h_s)_h$ strongly converges towards $r_s$ in $L^2(Q_T)$.
Using Point 7. in Proposition \ref{thm: global_bounds}, $(\theta^h_s)_h$ is uniformly bounded in $h$ in $L^2(Q_T)$, from which we get that $(\hat r^h_{ss})_h$ is uniformly bounded in $L^2(Q_T)$ as well, so it weakly converges towards a limit (up to extraction of a subsequence). Remembering that $\hat r^h_s \in L^2(0,T; H^1(0,L))$ converges towards $r_s$ in $L^2(Q_T)$, by uniqueness of the limit, $r_s$ is differentiable in $s$ and $\hat r^h_{ss}$ weakly converges towards $r_{ss}$ in $L^2(Q_T)$. \\

\noindent\textbf{Point 6. Convergence of $(\bar m^{h,z})_h$ in $L^2(Q_T)$.}
The convergence of $(\bar m^{h,z})_h$ towards $m^{z}$ follows from \eqref{eq: bornemi}, and the weak convergence of $(m^{h,z})_h$.\\

Of particular note, the extractions of subsequences may be done in a row, which means we can assume that the above-mentioned convergences are obtained for all sequences with the same extracted indices $(h_n)_{n\in \mathbb{N}}$ with $h_n \rightarrow 0$ as $n\rightarrow +\infty$.
\end{proof}

In order to prove Theorem \ref{thm: convergence}, it now remains to establish that the limits obtained in Proposition \ref{thm: extraction} are solutions of the continuous system \eqref{eq: continuous} in the sense of distributions:

\begin{proposition}[Limit equations]\label{thm: formefaible}
The limits $r \in H^1(Q_T)$, $m \in L^2(0,T; H^1(0,L))$ and $n \in L^2(0,T; H^1(0,L))$ of the interpolates satisfy the following system of equations: for any $\varphi \in C^{\infty}_c(0,L)$ and $\psi \in C^{\infty}_c(0,T)$,
\begin{subequations}\label{eq: continuous_weakform}
\begin{numcases}{}
\int_0^{T} \int_0^L \left(\C(r_s(t,s))\dot r(t,s) + n_s(t,s)\right)\varphi(s)  \psi(t) \,\d s\,\d t = 0,\label{eq: continuous_weakform_1}\\[6pt]
\displaystyle \int_0^{T} \int_0^L \left( m_s(t,s)  + r_s(t,s)\times n(t,s)\right) \varphi(s) \psi(t)\,\d s\, \d t = 0,\label{eq: continuous_weakform_2}\\[6pt]
\displaystyle \int_0^{T} \int_0^L m^z(t,s)\varphi(s)\psi(t)\,\d s\, \d t = \displaystyle E \int_0^{T} \int_0^L (r_{ss}(t,s) \times r_s(t,s))\cdot e_z \varphi(s)\psi(t)\,\d s\, \d t,\label{eq: continuous_weakform_3}\\[6pt]
 \displaystyle \left \vert r_s(t,s)\right\vert=1,\mbox{ for a.e. }(t,s)\in Q_T,\label{eq: continuous_weakform_4}\\[6pt]
 n(t,0) = n(t,L) = 0,\mbox{ for a.e. } t\in [0,T],\label{eq: continuous_weakform_5}\\[6pt]
 m(t,0) = m(t,L) =0,\mbox{ for a.e. } t\in [0,T].\label{eq: continuous_weakform_6}
\end{numcases}
\end{subequations}
\end{proposition}

\begin{proof}
 In order to proceed, let $\varphi \in C^{\infty}_c(0,L)$  and $\psi \in C^{\infty}_c(0,T)$ be given. Notice that, defining $\bar \varphi^h$ by, for $s \in [0,L]$
\begin{equation*}
    \bar \varphi^h(s) = \sum\limits_{i=1}^N \frac{1}{h} \left(\int_{L_i} \varphi(u)\d u\right) \mathbbm{1}_{L_i}(s).
    \label{eq: barvarphi}
\end{equation*}
we can rewrite System \eqref{eq: Nlink_bis_energ} as
\begin{subequations}\label{eq: Nlink_interp_system}
\begin{numcases}{}
\int_0^{T} \int_0^L \left(\C(r_s^h(t,s))\dot r^h(t,s) + n^h_s(t,s)\right)\bar \varphi^h(s)  \psi(t) \,\d s\,\d t = 0, \label{eq: Nlink_interp_system_1}\\[6pt]
\displaystyle \int_0^{T} \int_0^L \left(\left(-\frac{h^2}{12}\cn \dot{\bar{\theta}}^h(t,s) e_z + m_s^h(t,s) \right)\varphi(s) + \left(r_s^h(t,s)\times n^h(t,s) \bar\varphi^h(s)\right)\right) \psi(t)\,\d s\, \d t = 0,\label{eq: Nlink_interp_system_2}\\[6pt]
\displaystyle \int_0^{T} \int_0^L \bar{m}^{h,z}(t,s)\varphi(s)\psi(t)\,\d s\, \d t = E \displaystyle \int_0^{T} \int_0^L (\hat r^h_{ss}(t,s) \times \hat r^h_s(t,s))\cdot e_z \varphi(s)\psi(t)\,\d s\, \d t,\label{eq: Nlink_interp_system_3}\\[6pt]
 \displaystyle \left \vert r^h_s(t,s)\right\vert=1,\mbox{ for a.e. }(t,s)\in Q_T,\label{eq: Nlink_interp_system_4}\\[6pt]
 n^h(t,0) = n^h(t,L)= 0,\mbox{ for a.e. } t\in [0,T],\label{eq: Nlink_interp_system_5}\\[6pt]
 m^h(t,0)=  m^h(t,L) =0,\mbox{ for a.e. } t\in [0,T],\label{eq: Nlink_interp_system_6}
\end{numcases}
\end{subequations}
Then we pass to the limit $h\to 0$ in this discrete system.\\

First, notice that
\begin{equation}
\bar \varphi^h \rightarrow  \varphi \text{ strongly in } L^\infty(0,L).
\label{eq: cvPhi}
\end{equation}

Indeed, let $s\in [0,L]$ and $i$ such that $s\in L_i$. Then we have
\begin{eqnarray*}
    \left| \bar \varphi^h(s) - \varphi(s) \right|&=&\left| \frac 1 h \int_{L_i} (\varphi(u)-\varphi(s) ) \d u \right| \\
    &\leq&  \frac 1 h \int_{L_i} \left|\varphi(u)-\varphi(s) \right| \d u \\
    &\leq& h \|\varphi_s\|_\infty
\end{eqnarray*}
which proves the claim.

\noindent\textbf{Force balance.}
Since $(r_s^h)_h$ strongly converges to $r_s$ in $L^2(0,T; L^4(0,L))$, according to Point 2. in Proposition \ref{thm: extraction}, we deduce that
 $\C(r^h_s)$ strongly converges in $L^2(Q_T)$ to 
 $\C(r_s)$.
 Moreover, Points 1. and 4. in Proposition \ref{thm: extraction} also state that $(\dot r^h)_h$ weakly converges towards $\dot r$ in $L^2(Q_T)$, and $(n_s^h)_h$ weakly converges towards $n_s$ in $L^2(Q_T)$.
This is sufficient, using \eqref{eq: cvPhi}, to pass to the limit in equation \eqref{eq: Nlink_interp_system_1} and obtain \eqref{eq: continuous_weakform_1}.\\

\noindent\textbf{Moment balance.}
From Points 7. and 5. in Proposition \ref{thm: extraction}, we have that $h^2\dot{\bar{\theta}}^h$ strongly converges towards zero in $L^2(Q_T)$ and that $m^h_s$ also weakly converges to $m_s$ in $L^2(Q_T)$ respectively. The moment equation \eqref{eq: continuous_weakform_2} then follows from passing to the limit in \eqref{eq: Nlink_interp_system_2} using also \eqref{eq: cvPhi}.\\

\noindent\textbf{Inextensibility constraint.}
The limit $r_s \in L^2(0,T; L^4(0,L))$ of $r^h_s$ satisfies the inextensibility constraint \eqref{eq: continuous_weakform_4}, according to Point 2. in Proposition \ref{thm: extraction}, together with \eqref{eq: Nlink_interp_system_4}.\\

\noindent\textbf{Boundary conditions.}
We notice that both $(n^h)_h$ and $(m^h)_h$ are in $L^2(0,T;H^1_0(0,L))$ which is closed with respect to the weak $L^2(0,T;H^1(0,L))$ convergence. Therefore the weak limits $n$ and $m$ satisfy the boundary conditions \eqref{eq: continuous_weakform_5} and \eqref{eq: continuous_weakform_6}.\\

\noindent\textbf{Relation between $m$ and $\theta$.}
On the one hand, from Point 6. in Proposition \ref{thm: extraction} $(\bar m^{h,z})_h$ weakly converges to $m^z$ in $L^2(Q_T)$. On the other hand we know from Point 3. in Proposition \ref{thm: extraction} that 
\begin{eqnarray*}
\hat{r}^h_s &\rightarrow& r_s\mbox{ strongly in }L^2(Q_T),\\
\hat{r}^h_{ss} &\rightharpoonup& r_{ss}\mbox{ weakly in }L^2(Q_T).
\end{eqnarray*}
This is sufficient to pass to the limit in \eqref{eq: Nlink_interp_system_3}.
\end{proof}

In order to conclude the proof of Theorem \ref{thm: convergence}, we first remark that from the variational formulation \eqref{eq: continuous_weakform}, the equations of the continuous system \eqref{eq: continuous} are satisfied in a $L^2(Q_T)$ sense and therefore for a.e. $(s,t)\in Q_T$. It then remains to show that the initial conditions can also be deduced. But, we know that $r^h$ converges weakly to $r$ in $H^1(Q_T)$ which allows to deduce that $ r^h(0,\cdot)$ weakly converges to $r(0,\cdot)$ in $L^2(0,L)$ which, from \eqref{eq: cvCI}, gives $r(0,\cdot)=r^0$.

\begin{remark}
Note that Propositions \ref{thm: global_bounds}, \ref{thm: extraction} and \ref{thm: formefaible} do not use any assumption on boundary conditions in $s=0$. Therefore, they remain valid throughout regardless of choice of free, pinned or clamped boundary conditions.
\end{remark}

\subsection{Proof of Corollary~\ref{corollary}}

In this section we prove Corollary~\ref{corollary}, establishing the existence of a global solution to the continuous problem in suitable spaces. Let $r^0\in C^2(0,L)$ be given together with a representative $\theta^0\in C^1(0,L)$  such that $r^0_s = (\cos \theta^0, \, \sin \theta^0)^\intercal$.

Let $N\geq 1$ be an integer. From Theorem~\ref{thm: convergence}, it is sufficient to construct an initial condition $\Xt=(\theta_1, \ldots, \theta_N, r_1^{x,0},r_1^{y,0})$ for the $N$-link system such that the corresponding interpolant converges to $r^0$ in the sense of \eqref{eq: cvCI} and which satisfies the bound \eqref{eq: hyp_theta0}. 

In order to proceed, we consider $r^0_1=r^0(0)$, $\theta^0_i =\frac{1}{h}\int_{L_i} \theta^0(u)\d u$ and $r_i^0, 2 \leq i\leq N$ from \eqref{eq: Ri}. Let  $r^{h,0}$ and $\bar \theta^{h,0}$ be the corresponding interpolants given by \eqref{eq: interp_r_1} and \eqref{eq: interp_theta}.

On the one hand, we have 
\begin{align*}
     \Vert \theta^0 - \bar \theta^{h,0}  \Vert^2_{L^2(0,L)} &= \sum_{i=1}^N \int_{L_i} \vert \theta^0(s) - \theta_i^0 \vert^2 \d s\\
     & = \sum_{i=1}^N \int_{L_i} \bigg | \theta^0(s) - \frac{1}{h}\int_{L_i} \theta^0(u)\d u \bigg | ^2 \d s\\ 
     & \leq C h^2 \sum_{i=1}^N \int_{L_i} |\theta^0_s(s)| ^2 \d s\leq C' h^2.
     \label{eq: th0 - th0h}
\end{align*}
due to Poincaré-Wirtinger inequality and using $\theta^0\in C^1(0,L)$. Therefore, $\bar \theta^{h,0}$ converges strongly to $\theta^0$ in $L^2(0,L)$.

On the other hand, for the convergence of $r^{h,0}$, we first notice that $r^{h,0}-r^0$ vanish at 0 and satisfies
$$
|r_s^{h,0}-r^0_s| = |(\cos(\bar \theta^{h,0}) - \cos(\theta^0),\sin(\bar \theta^{h,0}) - \sin(\theta^0))^\intercal| \leq |\bar \theta^{h,0}-\theta^0|.
$$ 
This furnishes the bound
$$
\|r_s^{h,0}-r^0_s\|_{L^2(0,L)} \leq \|\bar \theta^{h,0}-\theta^0\|_{L^2(0,L)}
$$
which tends to 0 from the preceding calculation. This permits us to deduce that
$$
r^{h,0} \rightarrow r^0 \mbox{ strongly in } H^1(0,L),
$$
and therefore \eqref{eq: cvCI} holds.

We finish by proving the uniform bound \eqref{eq: hyp_theta0}. From the definition of $(\theta^0_i)_{1\leq i \leq N}$ we have for all $1\leq i\leq N-1$
\begin{eqnarray*}
    |\theta^0_{i+1}-\theta^0_i| &=& \frac{1}{h}\left|\int_0^h(\theta^0(ih+u)-\theta^0((i-1)h+u))\d u\right|\\
     &\leq& \frac{1}{h}\int_0^h \left|\theta^0(ih+u)-\theta^0((i-1)h+u)\right|\d u\\
    &\leq& h\|\theta^0_s\|_{L^\infty}
\end{eqnarray*}
which enables us to deduce
$$
\sum_{i=1}^{N-1}h\frac{|\theta^0_{i+1}-\theta^0_i|^2}{h^2} \leq Nh \|\theta^0_s\|^2_{L^\infty(0,L)} \leq L \|\theta^0_s\|^2_{L^\infty(0,L)}.
$$
Since $\theta^0\in C^1(0,L)$, we get the result.

\section{Discussion}\label{sec: discussion}

In this paper, we have addressed the mathematical validity of the $N$-link formulation for the elastohydrodynamics of a filament in a viscous flow. Theorems \ref{thm: discrete well-posedness}, \ref{thm: convergence} and Corollary \ref{corollary} respectively establish the existence and uniqueness of solutions to the $N$-link model, the convergence (up to extraction), in a weak sense, of \eqref{eq: Nlink_bis_energ} towards the classical filament elastohydrodynamics formulation \eqref{eq: continuous}, and the existence of a solution to the continuous system for smooth enough initial data. This constitutes a novel theoretical guarantee that the $N$-link model is a well-founded discretization of a continuous filament. 

The proofs of both theorems strongly rely on an energy dissipation formula \eqref{eq: energie}. The fact that this formula holds for System \eqref{eq: Nlink_bis_energ} and precisely leads to convergence underlines, in particular, the prevalence of the $N$-link model over more straightforward discretized versions of equation \eqref{eq: continuous}, such as ones stemming from a finite difference scheme on the filament arclength. 
Consequently, when numerically implementing filament dynamics, Theorem \ref{thm: convergence} justifies the use of an $N$-link filament as an equivalent system.

Further extensions of the discrete and continuous models to physically relevant cases and their potential difficulties to establish convergence are discussed in this section.

\paragraph{Non-local and non-Newtonian hydrodynamics.} It is worth recalling that the modeling of hydrodynamic interactions in both Systems \eqref{eq: continuous} and \eqref{eq: Nlink_bis_energ} are based on resistive force theory, which is a relatively coarse approximation retaining only local drag. 
Inclusion of nonlocal effects in the elastohydrodynamics equation may be challenging from the functional analysis point of view, as it requires to deal with integral terms in System \eqref{eq: Nlink_bis_energ}. 
On the other hand, the matrix formulation of the $N$-link model as stated in Equation \eqref{eq: matrixAXF} can handle the inclusion of nonlocal terms within the hydrodynamic matrix \cite{walker2020efficient}, while retaining the same structure. Then, establishing invertibility of the resistance matrix to extend the results of Theorem \ref{thm: discrete well-posedness} may still be tractable by following the same system reduction (Eq. \eqref{eq: BtildeF}).

Non-Newtonian fluids also constitute and important avenue of research for future extensions of this study, with many biological fluids exhibiting viscoelastic behaviors or complex rheologies \cite{spagnolie2023swimming}. Well-posedness for continuous elastohydrodynamics featuring an additional coupling modeling linear viscoelasticity was tackled by Ohm in \cite{ohm2024well}, whilst rigorous treatment of corresponding coarse-grained versions remains an open problem. 

\paragraph{Active filament.} Modeling internal activity within the curvature dynamics of the filament is a crucial step to fully describe the dynamics of biological filaments such as microswimmer flagella and cilia, rather than that of a passive elastic fiber for the present study. Activity is typically represented as an internal torque forcing, say $\tau$, which would appear in the moment equation (fifth line of System \eqref{eq: continuous}). By perturbative arguments, it is likely that well-posedness and convergence are preserved in the case of ``small enough'' $\tau$, following the proofs of Theorems \ref{thm: discrete well-posedness} and \ref{thm: convergence} and including suitable assumptions on the norm and regularity of $\tau$.  The general case for arbitrarily large $\tau$ is trickier, and probably does not admit more than local existence in time, as noted by Mori \& Ohm for continuous elastohydrodynamics \cite{mori_well-posedness_2023}.

\paragraph{Three-dimensional motion.} 
Whilst the flagellar waveform of many organisms, such as human sperm cells, is largely restricted to a two-dimensional space \cite{Lindemann}, other microswimmers like rodent spermatozoa \cite{Yanagimachi1970} and \textit{Escherichia coli} \cite{BERG1973} bacteria are known to exhibit out-of-plane deformation of their flagella.
Three-dimensional deformation of rods notoriously requires a more cumbersome description of the filament kinematics, with local deformation containing both bend and twist for inextensible filaments, and several more degrees of freedom in shear and compression in the general case of a Cosserat rod \cite{boyer2020dynamics}. 
In the context of 3D deformation of an elastic filament immersed in a Stokes flow, well-posedness of elastohydrodynamics equations in the lines of \cite[Theorem 1.1]{mori_well-posedness_2023} is currently being studied by Ohm \cite{ohm2025}.

Regarding modeling and simulation, three-dimensional coarse-grained models similar to the $N$-link swimmer with RFT-like approximations of the hydrodynamic drag are available \cite{garg2023slender,fuchter2023three} with the one developed by Walker \textit{et al.} \cite{walker2020efficient} being the most analogous formulation to the two-dimensional $N$-link model studied in the present paper. In particular, the dynamics of their model is governed by a matrix-vector system of equation with a similar structure to that of Equation \eqref{eq: matrixAXF}:
\begin{equation}
-\mathcal{B} \mathcal{A} \mathcal{Q} \dot \Theta = R,
\label{eq: N-link 3D}
 \end{equation}
where $\mathcal{B} \mathcal{A} \mathcal{Q}$ is square of dimension $(3N+3)\times(3N+3)$, and where blocks can be identified with the corresponding equation as in the $N$-link system (force or torque balance, constitutive equations). In this setting, $\mathcal{B}$ links the vector of torques and forces to a vector $R$, such that $-\mathcal{B}\begin{pmatrix}
    F & T \end{pmatrix}^\intercal = R$ represents the force and torque balance; while $\mathcal{A}$ links the vector of velocities $\dot X$ to the forces, and $\mathcal{Q}$ is a transition matrix from linear to angular velocities, i.e. $\mathcal{Q}\dot \Theta= \dot X$.
From there, one may hope to establish convergence of the solutions to \eqref{eq: N-link 3D} with the same strategy than in this paper: writing \eqref{eq: N-link 3D} as a fully differential equation to check well-posedness, obtaining bounds on interpolates in well-chosen functional spaces, and extracting convergent subsequences. 

\paragraph{Multiple filaments.}

Biological settings often feature several interacting filaments, whether it is seen in several monoflagellate organisms swimming together, multiple flagella beating in synchrony, or interconnected networks of actin filaments. 
In such cases, hydrodynamic interactions introduce coupling terms within the dynamics of each interacting filament, which may lead to synchronization, attraction and repulsion phenomena \cite{young2009hydrodynamic,du2019dynamics}.
To the best of our knowledge, coupled elastohydrodynamics equations have not been addressed from the point of view of well-posedness in the literature. Formally, in the continuous formulation of $M$ filaments interacting with each other through the surrounding fluid, coupling terms would appear as external forces $f$ and moments $l$ in System \eqref{eq: continuous antman} and within the boundary conditions, with one notable effect being the loss of translational and rotational invariance with respect to one filament's fixed frame. Similar additional terms would appear in the resistance matrices of $M$ coupled $N$-link filament models. Various approximations for these coupling hydrodynamics terms are available for modelling and simulation \cite{du2019dynamics}. From there, one could either consider each of the filaments as a single filament with ``unknown'' (that is, with no explicit expression available) source terms, hopefully possessing appropriate boundedness properties to preserve well-posedness; or study the whole coupled system of $M$ filaments -- a challenging task in that case being to check the invertibility of the as sociated resistance matrices. 

\printbibliography
\end{document}